\documentclass[12pt]{amsart}

\usepackage{pb-diagram}

\usepackage{amsfonts}
\usepackage{amsmath}
\usepackage{amssymb}

\newcommand{\Ric}[2]{R_{#1 \overline{#2}}}

\newtheorem{claim}{Claim}[section]
\newtheorem{theorem}{Theorem}[section]
\newtheorem{proposition}{Proposition}[section]
\newtheorem{lemma}{Lemma}[section]

\newtheorem{definition}{Definition}[section]
\newtheorem{corollary}{Corollary}[section]
\newtheorem{remark}{Remark}[section]

\def\Vol{{\rm Vol}}
\def\Fut{{\rm Fut}}
\def\PSH{{\rm PSH}}
\def\Ric{{\rm Ric}}
\def\Re{{\rm Re}}
\def\osc{{\rm osc}}
\def\Aut{{\rm Aut}}
\begin{document}

\title{Existence of K\"{a}hler-Ricci solitons on smoothable Q-Fano varieties }

\author{Yan Li}

\address{Institute of Mathematics, Hunan University, Changsha, China}

\email{liyandota@hotmail.com}

\thanks{}

\begin{abstract}
In this article we prove the existence of K\"{a}hler-Ricci solitons on smoothable, K-stable $\mathbb{Q}$-Fano varieties. We also investigate the behavior of twisted K\"{a}hler-Ricci solitons in the Gromov-Hausdorff topology under this smoothing family.
\end{abstract}

\maketitle

{\footnotesize  \tableofcontents}


\section{Introduction}
A basic problem in K\"{a}hler geometry is whether the Fano manifold $M$ admits a K\"{a}hler-Einstein metric. This problem is confirmed recently by Chen-Donaldson-Sun \cite{CDS1} \cite{CDS2} \cite{CDS3} and Tian \cite{T97} \cite{T15}, which claims that the existence of K\"{a}hler-Einstein metric on $M$ is equivalent to the algebro-geometric notion of K-stability. These techniques solving this problem play an important role in many other aspects. For instance, on one hand, this problem is reproved through the Aubin's continuity method developed by Datar-Sz\'{e}kelyhidi \cite{S1} \cite{DS}. Moreover, this continuity method is also adapted to deal with the problem that whether the Fano manifold $M$ admits a K\"{a}hler-Ricci soliton \cite{DS}. On the other hand, motivated by the study of the compactification of the moduli spaces of smooth K\"{a}hler-Einstein Fano manifold, Spotti-Sun-Yao \cite{SSY} investigate the existence of K\"{a}hler-Einstein metrics on smoothable $\mathbb{Q}$-Fano varieties by using the conic continuity method on a flat family. Combining these arguments, a natural problem is whether the existence of K\"{a}hler-Ricci soliton on a smoothable $\mathbb{Q}$-Fano variety $M$ is equivalent to the algebro-geometric notion of K-stable which is defined in \cite{DS}. It is notable that Berman-Nystr\"{o}m \cite{BN14} show that the existence of K\"{a}hler-Ricci soliton implies K-stable without any assumptions. Therefore, in this article we mainly consider the other side by applying the Aubin's continuity method on a flat family.

Before stating main results, we recall some basic definitions. A $\mathbb{Q}$-Fano variety $M$ is a normal projective variety with at worst log-terminal singularities and with ample $\mathbb{Q}$-Cartier anticanonical divisor $K_M^{-1}$. A $\mathbb{Q}$-Fano variety $M$ is called $\mathbb{Q}$-Gorestein smoothable if there is a flat projective family $\pi:\mathcal{M}\rightarrow \Delta$ over a disk $\Delta$ in $\mathbb{C}$ such that $M\cong M_0:=\pi^{-1}(0)$, $M_t:=\pi^{-1}(t)$ are smooth for $t\neq 0$ and $\mathcal{M}$ has a relatively ample $\mathbb{Q}$-Cartier anticanonical divisor $K_{\mathcal{M}/\Delta}^{-1}$. Proposition 1.41 \cite{KM} says that, by possibly shrinking $\Delta$, $M_t$ is a Fano manifold for $t\neq 0$ and there exists an integer $m>0$ such that $K_{M_t}^{-m}$ are very ample line bundles for all $t\in \Delta$. Let $\mathcal{V}$ be a holomorphic vector field on $\mathcal{M}$ which is only tangent to the fibers and belongs to a reductive algebra of reductive automorphism subgroup (c.f.\cite{Zhu00}) and $\mathcal{T}$ be the compact group induced by Im$\mathcal{V}$. Embed $\mathcal{M}$ into $\Delta\times \mathbb{CP}^N$ by using $\mathcal{T}$-invariant sections of $K_{\mathcal{M}/\Delta}^{-m}$ and denote $\alpha_t$ by the suitable scaling of the Fubuni-Study metric $\frac{1}{m}\omega_{FS}$ on $M_t$ for $t\in \Delta$.

Next we recall the definition of K-stable (c.f.\cite{DS}). Suppose that there exists a $\mathbb{C^*}$ action $\rho$ generated by a holomorphic vector field $W$ on $M_0$ which commutes with $V_0:=\mathcal{V}|_{M_0}$. Assume that $X:=\lim_{t\rightarrow 0}\rho(t)\cdot M_0$ is a $\mathbb{Q}$-Fano variety. We take the limit
$$
\alpha^*:=\lim_{t\rightarrow 0}\rho(t)\cdot \alpha_0, V^*:=\lim_{t\rightarrow 0}\rho(t)\cdot V_0.
$$
The $\mathbb{C^*}$ action $\rho$ defines a $T_0$-equivariant special degeneration ($T_0:=\mathcal{T}|_{M_0}$) and its twisted Futaki invariant is defined to be
\begin{align*}
\Fut_{(1-\lambda)\alpha_0,V_0}(M_0,W) &:=\Fut_{(1-\lambda)\alpha^*,V^*}(X,W_0)=\Fut_{V^*}(X,W_0)  \\
                                     &-\frac{1-\lambda}{\int_X \omega_\phi^n}\Big[\int_X
                                     \theta_{W_0}(e^{\theta_{V^*}}-1)\omega_\phi^n+n\int_X\theta_{W_0}
                                     (\alpha^*-\omega_\phi)\wedge \omega_\phi^{n-1} \Big],
\end{align*}
where $W_0$ is the induced vector field on $X$ by $W$, $\lambda\in(0,1]$, $\omega_\phi$ is the restriction of a suitable scaling of the Fubini-Study form on $X$, $\theta_{W_0}$ and $\theta_{V^*}$ are Hamiltonian functions and
$$
\Fut_{V^*}(X,W_0):=\frac{\int_X\theta_{W_0}e^{\theta_{V^*}} \omega_\phi^n}{\int_X \omega_\phi^n}.
$$

\begin{definition}
The triple $(M_0,(1-\lambda)\alpha_0,V_0)$ is K-semistable if $\Fut_{(1-\lambda)\alpha_0,V_0}(M_0,W)\geq 0$ for all $W$ as above. The triple is K-stable if in addition equality holds only when $(X,(1-\lambda)\alpha^*,V^*)$ is biholomorphic to $(M_0,(1-\lambda)\alpha_0,V_0)$.
\end{definition}

The main theorem of this article is the following result, which extends the consequences of \cite{DS}.
\begin{theorem}\label{T1.1}
Let $\pi:\mathcal{M}\rightarrow\Delta$ be a $\mathbb{Q}$-Gorestein smoothing of a $\mathbb{Q}$-Fano variety $M_0$ and $\mathcal{V}$ be a reductive holomorphic vector field on $\mathcal{M}$, which preserves the fibers. If $(M_0,V_0)$ is K-stable, then $M_0$ admits a K\"{a}hler-Ricci soliton.
\end{theorem}

We now briefly describe the structure of this article and sketch the main arguments needed to prove our main Theorem \ref{T1.1}. The strategy of the proof is based on Aubin's continuity method.

The first result, which is the subject of section 2 and 3, shows that there exists a unique twisted K\"{a}hler-Ricci soliton on a $\mathbb{Q}$-Fano variety when the parameter $\lambda$ close to $1-m^{-1}$. We mainly apply the pluripotential theory developed by Berman-Boucksom-Eyssidieux-Guedj-Zeriahi, see \cite{EGZ} \cite{BEGZ10} \cite{BBGZ13} \cite{BBEGZ} \cite{BN14} and \cite{VGBook}, to show that the properness of Mabuchi functional implies the existence and uniqueness of twisted K\"{a}hler-Ricci soliton. When $\lambda$ close to $1-m^{-1}$, the Mabuchi functional is proper due to the $\alpha$-invariant.

In section 4, we obtain the uniform $L^\infty$-estimate for the K\"{a}hler potentials. First, we show that if there exist twisted K\"{a}hler-Ricci solitons on $M_t$ for $t\in \Delta^*=\Delta\backslash\{0\}$ when $\lambda=\lambda_1$, then the Mabuchi functional has a uniformly lower bound for $\lambda<\lambda_1$. Note that when $\lambda<1-m^{-1}$, the Mabuchi functional is uniformly proper. Thus, the boundedness of the functional $I$ is obtained due to the fact that the Mabuchi functional is linear in $\lambda$. This is the reason why we take the parameter $\lambda\in(1-m^{-1},1]$ and $m>1$.

In section 5, we study the behavior of twisted K\"{a}hler-Ricci solitons in the Gromov-Hausdorff topology under the smoothing family. The arguments of \cite{RZ} and \cite{WZ} confirm the regular part of the Gromov-Hausdorff limit. This limit is homemorphic to the central fiber according to the arguments of \cite{DS1}.

The last result is that the following function
$$
\lambda_t:=\sup\{\lambda\in(1-m^{-1},1]|\exists \textrm{twisted K\"{a}hler-Ricci soliton on $M_t$ for all $\kappa\leq \lambda$}\}
$$
is lower semi-continuous, which we present in section 6. It implies the openness and closedness in the Aubin's continuity method.

Next we give some remarks on the main Theorem \ref{T1.1}. First, the main technical point where the smoothability is used here is given by the application of smooth Riemannian convergence theory with Bakry-\'{E}mery Ricci curvature bounded below which is developed by Wang-Zhu \cite{WZ} and Datar-Sz\'{e}kelyhidi \cite{DS}. Second, we expect that Theorem \ref{T1.1} holds for general, not necessarily smoothable, $\mathbb{Q}$-Fano varieties. But it is difficult. From now on, Li-Tian-Wang \cite{LTW}
show that the existence of weak K\"{a}hler-Einstein metric is equivalent to the algebraic notion K-stability on a $\mathbb{Q}$-Fano variety with admissible singularities. The admissible singularities imply that the metrics they deal with always have at worst conic singularities. Thus, this problem is still open for general case. It is notable that recently Li \cite{LC19} claims that the uniform K-stability is equivalent to the existence of weak K\"{a}hler-Einstein metric on a $\mathbb{Q}$-Fano variety without any assumptions by applying the argument of \cite{BBJ}.

There are also fundamental results about the moduli spaces of smooth K\"{a}hler-Einstein manifolds, see \cite{RZ} \cite{S3} \cite{LXW} \cite{SP1} \cite{SP2} \cite{OSS} and \cite{O}. This is another motivation for this article.

\section{Preliminaries}
In this section we will establish some elementary estimates which will be used in the later. Let $M_0$ be a $\mathbb{Q}$-Fano variety and $V_0$ be a reductive holomorphic vector field defined on the regular part of $M_0$. If $\pi:M\rightarrow M_0$ is a log resolution, by normality, the vector field $V_0$ admits a unique extension $V$ to $M$ (c.f. section 2.3 \cite{BN14}). Denote $T$ and $T_0$ by the compacts groups induced by Im$V$ and Im$V_0$. There exists an integer $m>0$ such that $M_0$ can be embedded into $\mathbb{CP}^N$ by using $T_0$-invariant sections of $K_{M_0}^{-m}$. $\alpha_0$ denotes the scaling of the Fubini-Study form $\frac{1}{m}\omega_{FS}$. Set $\omega_0=\pi^*\alpha_0$, then $\omega_0$ is $T$-invariant. We introduce
$$
\PSH(M,\omega_0):=\{\varphi|\omega_0+\sqrt{-1}\partial\bar{\partial}\varphi\geq 0\}
$$
and
$$
\PSH(M,\omega_0)^T:=\{\varphi|\omega_0+\sqrt{-1}\partial\bar{\partial}\varphi\geq 0 \textrm{\ and $\varphi$ is $T$-invariant}\}.
$$
\begin{lemma}\label{L2.1}
If $\varphi\in \PSH(M,\omega_0)^T$, then $V(\varphi)$ is well-defined and $|V(\varphi)|\leq C$ a.e. $[\chi^n]$, where $C$ is a constant independent of $\varphi$ and $\chi$ is a $T$-invariant K\"{a}hler metric on $M$.
\end{lemma}
\begin{proof}
There is a strictly decreasing sequence $\varphi_j$ of smooth functions with limit $\varphi$ such that $\omega_0+\epsilon_j\chi+\sqrt{-1}\partial\bar{\partial}\varphi_j>0$ and $\epsilon_j$ decreases to $0$ due to \cite{BK07}. By averaging we can assume that $\varphi_j$ are $T$-invariant. Lemma 5.1 \cite{Zhu00} and Corollary 5.3 \cite{Zhu00} imply that $|V(\varphi_j)|\leq C$, where $C$ is a constant independent of $\varphi_j$. By Theorem 1.48 \cite{VGBook} and locally argument, $\nabla\varphi_j$ converge to $\nabla\varphi$ in $L^q$ for all $1\leq q<2$, where $\nabla$ denotes the gradient of functions. Furthermore, there exists a subsequence $j_k$ such that $\nabla\varphi_{j_k}$ converges to $\nabla\varphi$ a.e. $[\chi^n]$. So $V(\varphi)$ is well-defined and $|V(\varphi)|\leq C$.
\end{proof}
In \cite{BEGZ10}, the finite energy class
$$
\PSH_{full}(M,\omega_0):=\Big\{\varphi\in \PSH(M,\omega_0)\Big|\int_M (\omega_0+\sqrt{-1}\partial\bar{\partial}\varphi)^n=\int_M \omega_0^n=a\Big\}
$$
has been investigated. Similarly, we need the following definition.
\begin{definition}
The $T$-invariant finite energy class is
$$
\PSH_{\textrm{full}}(M,\omega_0)^T:=\Big\{\varphi\in \PSH(M,\omega_0)^T\Big|\int_M (\omega_0+\sqrt{-1}\partial\bar{\partial}\varphi)^n=\int_M \omega_0^n=a\Big\}
$$
\end{definition}
\begin{lemma}\label{L2.2}
If $\varphi\in \PSH(M,\omega_0)^T$ and $\psi\in \PSH_{\textrm{full}}(M,\omega_0)^T$, then $|V(\varphi)|\leq C$ a.e. $[(\omega_0+\sqrt{-1}\partial\bar{\partial}\psi)^n]$.
\end{lemma}
\begin{proof}
By the Lemma \ref{L2.1}, there exists a constant $C$ independent of $\varphi$ such that $|V(\varphi)|\leq C$ a.e. $[\chi^n]$. We introduce the set $S:=\{x\in M||V(\varphi)(x)|>C\}$, then there is a Borel set $B\supset S$ which is $G_\delta$ such that $\int_B \chi^n=0$. We take the canonical approximation $\psi_j:=\max(\psi,-j)$. Proposition 10.15 \cite{VGBook} claims that $\lim_{j\rightarrow\infty}\int_B(\omega_0+\sqrt{-1}\partial\bar{\partial}\psi_j)^n=\int_B(\omega_0+\sqrt{-1}\partial\bar{\partial}\psi)^n$.
For each $j$, if $\int_B(\omega_0+\sqrt{-1}\partial\bar{\partial}\psi_j)^n=0$, then this lemma is true. Next we assume that $\psi\in \PSH_{\textrm{full}}(M,\omega_0)^T\cap L^\infty(M)$. Choosing a decreasing sequence $\psi^k$ of smooth functions with limit $\psi$ such that $\omega_0+\epsilon_k\chi+\sqrt{-1}\partial\bar{\partial}\psi^k>0$ and $\epsilon_k$ decreases to $0$. Theorem 3.18 \cite{VGBook} implies that $\lim_{k\rightarrow\infty}\int_B(\omega_0+\epsilon_k\chi+\sqrt{-1}\partial\bar{\partial}\psi^k)^n=\int_B(\omega_0+\sqrt{-1}\partial\bar{\partial}\psi)^n$.
Note that $\int_B(\omega_0+\epsilon_k\chi+\sqrt{-1}\partial\bar{\partial}\psi^k)^n=0$ for each $k$ since $\int_B\chi^n=0$. Therefore the proof is completed.
\end{proof}
$\PSH_{{full}}(M,\omega_0)^T$ is convex according to the same argument of Proposition 10.7 \cite{VGBook}. Next we introduce the following functionals.
\begin{definition}
For $\phi\in \PSH(M,\omega_0)^T\cap L^\infty(M)$,
$$
E_V(\phi):=\int_0^1\int_M\phi e^{\theta_M+s\cdot V(\phi)}\omega^n_{s\phi}\wedge ds
$$
and
$$
E(\phi):=\frac{1}{n+1}\sum_{j=0}^n\int_M\phi\omega_\phi^j\wedge\omega_0^{n-j}
$$
where $\omega_{s\phi}:=\omega_0+\sqrt{-1}\partial\bar{\partial}(s\phi)$ and $\theta_M$ is defined by
$L_V\omega_0=\sqrt{-1}\partial\bar{\partial}\theta_M$.
\end{definition}
\begin{definition}
For $\varphi\in \PSH(M,\omega_0)^T$,
$$
E(\varphi):=\inf\{E(\phi)|\phi\geq \varphi,\phi\in \PSH(M,\omega_0)\cap L^\infty(M)\},
$$
$$
E_V(\varphi):=\inf\{E_V(\phi)|\phi\geq \varphi,\phi\in \PSH(M,\omega_0)^T\cap L^\infty(M)\},
$$
$$
\mathcal{E}^1(M,\omega_0):=\{\varphi\in \PSH_{\textrm{full}}(M,\omega_0)|E(\varphi)>-\infty \},
$$
$$
\mathcal{E}^1_V(M,\omega_0):=\{\varphi\in \PSH_{\textrm{full}}(M,\omega_0)^T|E_V(\varphi)>-\infty\}.
$$
\end{definition}
The following lemma is standard due to Proposition 2.15 \cite{BN14}.
\begin{lemma}\label{L2.3}
The map $\varphi\mapsto E_V(\varphi)$ is upper semi-continuous for the $L^1$-topology on $\PSH(M,\omega_0)^T$. It is continuous along decreasing sequences in $\PSH(M,\omega_0)^T$.
\end{lemma}
\begin{lemma}\label{L2.4}
If $\varphi\in \mathcal{E}^1_V(M,\omega_0)$, then $\varphi\in \mathcal{E}^1(M,\omega_0)$.
\end{lemma}
\begin{proof}
Define $\varphi_j:=\max(\varphi,-j)$, then Lemma \ref{L2.3} and Proposition 10.19 \cite{VGBook} imply that $\lim_{j\rightarrow\infty}E_V(\varphi_j)=E_V(\varphi)$ and $\lim_{j\rightarrow\infty}E(\varphi_j)=E(\varphi)$. Assume that $\varphi\leq C_1$ where $C_1$ is a positive constant, we have the following calculations
\begin{align*}
E_V(\varphi_j)&=\int_0^1\int_M(\varphi_j-C_1) e^{\theta_M+s\cdot V(\varphi_j)}\omega^n_{s\varphi_j}\wedge ds+\int_0^1\int_M C_1 e^{\theta_M+s\cdot V(\varphi_j)}\omega^n_{s\varphi_j}\wedge ds\\
&\geq e^C\int_0^1\int_M(\varphi_j-C_1)\omega^n_{s\varphi_j}\wedge ds+C_1e^{-C}\int_0^1\int_M\omega^n_{s\varphi_j}\wedge ds\\
&= e^CE(\varphi_j)+C_1(e^{-C}-e^C)a
\end{align*}
where the second inequality holds due to Lemma \ref{L2.2}. By the same argument, we have
\begin{align*}
E_V(\varphi_j)&\leq e^{-C}\int_0^1\int_M(\varphi_j-C_1)\omega^n_{s\varphi_j}\wedge ds+C_1e^C\int_0^1\int_M\omega^n_{s\varphi_j}\wedge ds\\
&=e^{-C}E(\varphi_j)+C_1(e^C-e^{-C})a.
\end{align*}
Taking limit on both sides
$$
e^CE(\varphi)+C_1(e^{-C}-e^C)a\leq E_V(\varphi)\leq e^{-C}E(\varphi)+C_1(e^C-e^{-C})a.
$$
Therefore, $E_V(\varphi)>-\infty$ gives $E(\varphi)>-\infty$.
\end{proof}
\begin{proposition}\label{P2.1}
Let $\varphi\in \PSH(M,\omega_0)^T$ and $\varphi_j:=\max(\varphi,-j)$. Assume that $V(\varphi_j)$ pointwise converges to $V(\varphi)$, then
$$
e^{\theta_M+V(\varphi_j)}\omega^n_{\varphi_j}\rightarrow e^{\theta_M+V(\varphi)}\omega^n_{\varphi}
$$
weakly as $j\rightarrow\infty$.

If $\varphi\in \mathcal{E}^1_V(M,\omega_0)$ and $V(\varphi_j)$ pointwise converges to $V(\varphi)$, then
$$
\varphi_je^{\theta_M+V(\varphi_j)}\omega^n_{\varphi_j}\rightarrow \varphi e^{\theta_M+V(\varphi)}\omega^n_{\varphi}
$$
weakly as $j\rightarrow\infty$.
\end{proposition}
\begin{proof}
The first argument is obtained according to Theorem 2.7 \cite{BN14}. Next we prove the second argument which is similar as Theorem 2.17 \cite{BEGZ10}. Let $h$ be a continuous function on $M$, then it is enough to establish that
$$
\lim_{j\rightarrow\infty}\int_M h\varphi_je^{\theta_M+V(\varphi_j)}\omega^n_{\varphi_j}=\int_M h\varphi e^{\theta_M+V(\varphi)}\omega^n_{\varphi}.
$$
We have
\begin{align*}
&\Big|\int_M h(\varphi_je^{\theta_M+V(\varphi_j)}\omega^n_{\varphi_j}-\varphi e^{\theta_M+V(\varphi)}\omega^n_{\varphi})\Big|\leq \int_{\{\varphi>-j\}}|h||\varphi|\big|e^{\theta_M+V(\varphi_j)}-e^{\theta_M+V(\varphi)}\big|\omega^n_{\varphi}\\
& +\int_{\{\varphi\leq -j\}}|h||\varphi_j|e^{\theta_M+V(\varphi_j)}\omega^n_{\varphi_j}+\int_{\{\varphi\leq -j\}}|h||\varphi|e^{\theta_M+V(\varphi)}\omega^n_{\varphi}.
\end{align*}
The condition that $V(\varphi_j)$ pointwise converges to $V(\varphi)$ implies
$$
\lim_{j\rightarrow\infty}\int_{\{\varphi>-j\}}|h||\varphi|\big|e^{\theta_M+V(\varphi_j)}-e^{\theta_M+V(\varphi)}\big|\omega^n_{\varphi}
\leq
\lim_{j\rightarrow\infty}\int_M|h||\varphi|\big|e^{\theta_M+V(\varphi_j)}-e^{\theta_M+V(\varphi)}\big|\omega^n_{\varphi}=0.
$$

Lemma \ref{L2.4} and Exercise 10.5 \cite{VGBook} show that $\varphi\in \mathcal{E}^1(M,\omega_0)$ and there exists a convex weight $\gamma$ such that $\lim_{k\rightarrow\infty}\frac{-k}{\gamma(-k)}=0$ and $\int_M \gamma(\varphi)\omega^n_\varphi>-\infty$, where a weight denotes a smooth increasing function $\gamma:\mathbb{R}\rightarrow\mathbb{R}$ such that $\gamma(-\infty)=-\infty$. According to Lemma \ref{L2.2}, we have
\begin{align*}
\int_{\{\varphi\leq -j\}}|h||\varphi_j|e^{\theta_M+V(\varphi_j)}\omega^n_{\varphi_j}& \leq \sup_M|h|e^C\int_{\{\varphi\leq -j\}}|\varphi_j|\omega^n_{\varphi_j} \\
& =\sup_M|h|e^C\int_{\{\varphi\leq -j\}}|\gamma(\varphi_j)|\cdot\frac{|\varphi_j|}{|\gamma(\varphi_j)|}\omega^n_{\varphi_j}\\
&\leq \sup_M|h|e^C\cdot\frac{-j}{\gamma(-j)}\cdot\int_M|\gamma(\varphi_j)|\omega^n_{\varphi_j},
\end{align*}
which yields
$$
\lim_{j\rightarrow\infty}\int_{\{\varphi\leq -j\}}|h||\varphi_j|e^{\theta_M+V(\varphi_j)}\omega^n_{\varphi_j}=0.
$$
Also
\begin{align*}
\int_{\{\varphi\leq -j\}}|h||\varphi|e^{\theta_M+V(\varphi)}\omega^n_{\varphi}&\leq \sup_M|h|e^C\cdot
\lim_{k\rightarrow\infty}\int_{\{-k<\varphi\leq j\}}|\varphi_k|\omega^n_{\varphi_k}\\
&\leq \sup_M|h|e^C\cdot\frac{-j}{\gamma(-j)}\cdot\limsup_{k\rightarrow\infty}\int_M|\gamma(\varphi_k)|\omega_{\varphi_k}^n,
\end{align*}
which gives
$$
\lim_{j\rightarrow\infty}\int_{\{\varphi\leq -j\}}|h||\varphi|e^{\theta_M+V(\varphi)}\omega^n_{\varphi}=0.
$$
This proposition is proved.
\end{proof}
\begin{corollary}
If $\varphi\in \mathcal{E}^1_V(M,\omega_0)$, then
$$
E_V(\varphi)=\int_0^1\int_M\varphi e^{\theta_M+s\cdot V(\varphi)}\omega^n_{s\varphi}\wedge ds.
$$
\end{corollary}
\begin{proof}
Set $\varphi_j:=\max(\varphi,-j)$, $f_j(s):=\int_M\varphi_je^{\theta_M+s\cdot V(\varphi_j)}\omega^n_{s\varphi_j}$ and $f(s):=\int_M\varphi e^{\theta_M+s\cdot V(\varphi)}\omega^n_{s\varphi}$, then Proposition \ref{P2.1} shows that $\lim_{j\rightarrow \infty}f_j(s)=f(s)$ for each $s\in [0,1]$. Note that
\begin{align*}
|f_j(s)|& \leq e^C\int_M|\varphi_j|((1-s)\omega_0+s\omega_{\varphi_j})^n
=e^C\int_M|\varphi_j|\cdot\sum_{m=0}^n\textrm{C}^n_ms^{n-m}(1-s)^m\omega_0^m\wedge \omega^{n-m}_{\varphi_j}\\
& \leq C'\int_M|\varphi_j|\omega_{\varphi_j}^n\leq C'',
\end{align*}
where $C'$ and $C''$ are positive constants and the third inequality bases on $\varphi\in \mathcal{E}^1(M,\omega_0)$. By the Lebesgue dominated convergence theorem, $\lim_{j\rightarrow \infty}\int_0^1f_j(s)ds=\int_0^1f(s)ds$.
\end{proof}
\begin{lemma}\label{L2.5}
Let $\mathcal{E}^1_{V,C}(M,\omega_0):=\{\varphi\in \mathcal{E}^1_{V}(M,\omega_0)|E_V(\varphi)\geq -C \ \ \textrm{and} \ \ \sup_M\varphi\leq 0\}$, then it is a compact subset for the $L^1$-topology.
\end{lemma}
\begin{proof}
For $\varphi\in \mathcal{E}^1_{V,C}(M,\omega_0)$, we see
$$
-C\leq E_V(\varphi)=\int_0^1\int_M\varphi e^{\theta_M+s\cdot V(\varphi)}\omega^n_{s\varphi}\wedge ds\leq e^{-C}\cdot a\cdot (\sup_M\varphi).
$$
So there exists a constant $C>0$ independent of $\varphi$ such that $-C\leq \sup_M \varphi\leq 0$, which implies
$$
\mathcal{E}^1_{V,C}(M,\omega_0)\subset \{\varphi\in \PSH(M,\omega_0)^T|-C\leq \sup_M \varphi\leq 0\}.
$$
The latter set is a compact subset of $\PSH(M,\omega_0)^T$ by Hartog's Lemma, see Theorem 1.46 \cite{VGBook}. Since $\varphi\mapsto E_V(\varphi)$ is upper semi-continuous by Lemma \ref{L2.3}, the set $\mathcal{E}^1_{V,C}(M,\omega_0)$ is closed, hence compact for $L^1$-topology.
\end{proof}
To deal with K\"{a}hler-Ricci soliton, the following functionals are introduced (c.f.\cite{TZ00}). For $\phi\in \mathcal{E}^1_V(M,\omega_0)$, we define
$$
I_V(\phi)=\int_M\phi(e^{\theta_M}\omega^n_0-e^{\theta_M+V(\phi)}\omega_{\phi}^n)
$$
and
$$
J_V(\phi)=\int_0^1\int_M\phi(e^{\theta_M}\omega_0^n-e^{\theta_M+s\cdot V(\phi)}\omega_{s\phi}^n)\wedge ds.
$$
\begin{proposition}\label{P2.2}
Define $\alpha_M=\inf_M \theta_M$ and $\beta_M=\sup_M\theta_M$ which are independent of the choice of $\omega_0$, the we have
$$
I_V(\phi)\leq (n+1+\beta_M-\alpha_M)(I_V(\phi)-J_V(\phi))\leq (n+\beta_M-\alpha_M)I_V(\phi).
$$
\end{proposition}
\begin{proof}
Taking $\phi_j:=\max(\phi,-j)$, then $\lim_{j\rightarrow \infty}I_V(\phi_j)=I_V(\phi)$ and $\lim_{j\rightarrow \infty}J_V(\phi_j)=J_V(\phi)$ according to Proposition \ref{P2.1} when $V(\phi_j)$ pointwise converges to $V(\phi)$. Without loss of generality, we can assume that $\phi\in \mathcal{E}^1_V(M,\omega_0)\cap L^\infty(M)$. By the approximation theorem \cite{BK07}, there is a strictly decreasing sequence $\phi^k$ of smooth functions with limit $\phi$ such that $\omega_0+\epsilon_k\chi+\sqrt{-1}\partial\bar{\partial}\phi^k>0$. We further assume that $\phi^k$ are $T$-invariant and $V(\phi^k)$ pointwise converges to $V(\phi)$. Define $\theta_{M,k}$ by $L_V(\omega_0+\epsilon_k\chi)=\sqrt{-1}\partial\bar{\partial}\theta_{M,k}$ and $\alpha_{M,k}:=\inf_M \theta_{M,k}$,  $\beta_{M,k}:=\sup_M\theta_{M,k}$. We denote $\omega_k$ by $\omega_0+\epsilon_k\chi$ and define
$$
I_V(\phi^k)=\int_M\phi^k(e^{\theta_{M,k}}\omega^n_k-e^{\theta_{M,k}+V(\phi^k)}\omega^n_{\phi^k})
$$
and
$$
J_V(\phi^k)=\int_0^1\int_M\phi^k(e^{\theta_{M,k}}\omega^n_k-e^{\theta_{M,k}+s\cdot V(\phi^k)}\omega^n_{s\phi^k})\wedge ds
$$
where $\omega_{s\phi^k}:=\omega_k+\sqrt{-1}\partial\bar{\partial}(s\phi^k)$. Proposition A.1 \cite{M1} implies that
\begin{equation}\label{e2.1}
I_V(\phi^k)\leq (n+1+\beta_{M,k}-\alpha_{M,k})(I_V(\phi^k)-J_V(\phi^k))\leq (n+\beta_{M,k}-\alpha_{M,k})I_V(\phi^k).
\end{equation}
By Theorem 2.7 \cite{BN14} and the fact that $\lim_{k\rightarrow\infty}\theta_{M,k}=\theta_M$, we know that $\lim_{k\rightarrow \infty}I_V(\phi^k)=I_V(\phi)$ and $\lim_{k\rightarrow \infty}J_V(\phi^k)=J_V(\phi)$. Therefore, by taking the limit on inequality (\ref{e2.1}) we deduce this proposition.
\end{proof}
\section{The variational approach for twisted K\"{a}hler-Ricci solitons}
This section is devoted to explain a variational approach developed in \cite{BBGZ13} to solve the twisted K\"{a}hler-Ricci soliton equation.

Recall that if $\pi:M\rightarrow M_0$ is a log resolution, then there exist rational numbers $a_i\geq 0$ and $0<b_j<1$ with
$$
K_M=\pi^*K_{M_0}+\sum_ia_iE_i-\sum_jb_jF_j
$$
where $E_i$ and $F_j$ are exceptional prime divisors. We embed $M_0$ into $\mathbb{CP}^N$ by using $T_0$-invariant sections of $K_{M_0}^{-m}$. $\alpha_0$ denotes $\frac{1}{m}\omega_{FS}$. Let $\nu$ be an adapted measure with $\sqrt{-1}\partial\bar{\partial}\log \nu=-\alpha_0$ on $(M_0)_{\textrm{reg}}$, where
$(M_0)_{\textrm{reg}}$ denotes the regular part of $M_0$. $\theta_{M_0}$ is a Hamiltonian function defined by $L_{V_0}\alpha_0=\sqrt{-1}\partial\bar{\partial}\theta_{M_0}$.
\begin{definition}\label{D3.1}
For $\lambda\in (1-m^{-1},1]$, a twisted K\"{a}hler-Ricci soliton for the triple $(M_0,V_0,(1-\lambda)\omega_{FS})$ is a current $\omega_\phi:=\alpha_0+\sqrt{-1}\partial\bar{\partial}\phi$ with full Monge-Amp\`{e}re mass, i.e. $\phi\in \PSH_{full}(M_0,\alpha_0)^{T_0}$ such that
$$
e^{\theta_{M_0}+V_0(\phi)}(\alpha_0+\sqrt{-1}\partial\bar{\partial}\phi)^n=\frac{e^{-r(\lambda)\phi}\nu}{\int_{M_0}e^{-r(\lambda)\phi}\nu}
$$
where we assume that $e^{\theta_{M_0}+V_0(\phi)}(\alpha_0+\sqrt{-1}\partial\bar{\partial}\phi)^n$ is a probability measure on $M_0$ and $r(\lambda)=1-(1-\lambda)m$.
\end{definition}
\begin{remark}
The existence of the twisted K\"{a}hler-Ricci soliton is also equivalent to solve the following degenerated complex Monge-Amp\`{e}re equation on $M$
\begin{equation}
e^{\theta_{M}+V(\phi)}(\omega_0+\sqrt{-1}\partial\bar{\partial}\phi)^n=\frac{e^{-r(\lambda)\phi}\mu}{\int_Me^{-r(\lambda)\phi}\mu},
\end{equation}\label{e3.1}
where $\omega_0=\pi^*\alpha_0$, $\mu=\pi^*\nu$ and $\theta_M$ is a Hamiltonian function defined by $L_V\omega_0=\sqrt{-1}\partial\bar{\partial}\theta_M$.
\end{remark}
Next, some consequences about $\alpha$-invariant defined by Tian \cite{T87} will be recalled (c.f. \cite{BBEGZ}).
\begin{definition}
The $\alpha$-invariant of a measure $\mu$ is defined as
$$
\alpha_\mu(\omega_0):=\sup\Big\{\alpha>0\Big|\sup_{\varphi\in \PSH(M,\omega_0)}\int_Me^{-\alpha\varphi}d\mu<+\infty\Big\}.
$$
\end{definition}
\begin{remark}\label{r3.2}
The $\alpha$-invariant $\alpha_\mu(\omega_0)>0$ due to Proposition 1.4 \cite{BBEGZ}.
\end{remark}
Mabuchi functional and Ding functional play important roles in the research of the existence of K\"{a}hler-Einstein metrics on Fano manifolds. Similarly, we need the following
\begin{definition}
For $\varphi\in \mathcal{E}^1_V(M,\omega_0)$, we define the twisted Ding functional to be
\begin{align*}
F_{V,\lambda}(\varphi)&=-r(\lambda)\cdot E_V(\varphi)-\log\int_Me^{-r(\lambda)\varphi}d\mu\\
&=-r(\lambda)\cdot\int_0^1\int_M\varphi e^{\theta_M+s\cdot V(\varphi)} \omega^n_{s\varphi}\wedge ds-\log\int_Me^{-r(\lambda)\varphi}d\mu
\end{align*}
\end{definition}
\begin{proposition}\label{P3.1}
$F_{V,\lambda}$ is lower semi-continuous on each $\mathcal{E}_{V,C}^1(M,\omega_0)$ defined in Lemma \ref{L2.5}.
\end{proposition}
\begin{proof}
Proposition 11.3 (iii) \cite{VGBook} implies that $\varphi\mapsto \log\int_Me^{-r(\lambda)\varphi}d\mu$ is continuous on $\mathcal{E}_{V,C}^1(M,\omega_0)$. The conclusion follows due to Lemma \ref{L2.3}.
\end{proof}
\begin{definition}
For $\varphi\in \mathcal{E}^1_V(M,\omega_0)$, we define the twisted Mabuchi functional to be
$$
M_{V,\lambda}(\varphi)=-r(\lambda)\cdot \Big(E_V(\varphi)-\int_M\varphi e^{\theta_M+V(\varphi)} \omega^n_{\varphi}\Big)+\int_M\log\frac{e^{\theta_M+V(\varphi)} \omega^n_{\varphi}}{\mu}e^{\theta_M+V(\varphi)} \omega^n_{\varphi}.
$$
\end{definition}
Set $$\mu_\varphi=\frac{e^{-r(\lambda)\varphi}\mu}{\int_Me^{-r(\lambda)\varphi}d \mu},$$ we have
\begin{lemma}\label{L3.1}
For $\varphi\in \mathcal{E}^1_V(M,\omega_0)$,
$$
F_{V,\lambda}(\varphi)=M_{V,\lambda}(\varphi)-\int_M\log\frac{e^{\theta_M+V(\varphi)} \omega^n_{\varphi}}{\mu_\varphi}e^{\theta_M+V(\varphi)} \omega^n_{\varphi}\leq M_{V,\lambda}(\varphi).
$$
\end{lemma}
\begin{proof}
Observing that
\begin{align*}
& \int_M\log\frac{e^{\theta_M+V(\varphi)} \omega^n_{\varphi}}{\mu_\varphi}e^{\theta_M+V(\varphi)} \omega^n_{\varphi}=\int_M\log\frac{e^{\theta_M+V(\varphi)} \omega^n_{\varphi}}{\mu}e^{\theta_M+V(\varphi)} \omega^n_{\varphi}\\
& \qquad +r(\lambda)\cdot\int_M\varphi e^{\theta_M+V(\varphi)}\omega_\varphi^n+\log\int_Me^{-r(\lambda)\varphi}d\mu.
\end{align*}
By definition, we have
$$
 M_{V,\lambda}(\varphi)-\int_M\log\frac{e^{\theta_M+V(\varphi)} \omega^n_{\varphi}}{\mu_\varphi}e^{\theta_M+V(\varphi)} \omega^n_{\varphi}=F_{V,\lambda}(\varphi).
$$
Jensen's inequality implies that
$$
\int_M\log\frac{e^{\theta_M+V(\varphi)} \omega^n_{\varphi}}{\mu_\varphi}e^{\theta_M+V(\varphi)} \omega^n_{\varphi}\geq 0.
$$
Thus the proof is completed.
\end{proof}
\begin{definition}
We say that the functional $M_{V,\lambda}$ $(F_{V,\lambda})$ is proper if whenever $\varphi_j\in \mathcal{E}^1_V(M,\omega_0)$ is a sequence of functions such that $J_V(\varphi_j)\rightarrow+\infty$, then $M_{V,\lambda}(\varphi_j)\rightarrow+\infty$ $(F_{V,\lambda}(\varphi_j)\rightarrow+\infty)$.
\end{definition}
\begin{lemma}\label{L3.2}
Fix $0<\sigma<\alpha_\mu(\omega_0)$. There exists a constant $C_\sigma$ such that
$$
M_{V,\lambda}(\varphi)\geq \Big(\sigma-r(\lambda)\cdot\frac{n+\beta_M-\alpha_M}{n+1+\beta_M-\alpha_M}\Big)I_V(\varphi)-C_\sigma
$$
for all $\varphi\in \mathcal{E}^1_V(M,\omega_0)$ and $\sup_M\varphi=0$, where $\alpha_M$ and $\beta_M$ are defined in Proposition \ref{P2.2}. In particular, if $\alpha_\mu(\omega_0)>r(\lambda)\cdot\frac{n+\beta_M-\alpha_M}{n+1+\beta_M-\alpha_M}$, then $M_{V,\lambda}$ is proper.
\end{lemma}
\begin{proof}
By assumption
$$
\int_M e^{-\sigma\varphi-\log\frac{e^{\theta_M+V(\varphi)}\omega^n_\varphi}{\mu}}\cdot e^{\theta_M+V(\varphi)}\omega^n_\varphi=\int_Me^{-\sigma\varphi}d\mu\leq e^{C_\sigma}.
$$
Jensen's inequality implies that
$$
-\sigma\cdot\int_M\varphi e^{\theta_M+V(\varphi)}\omega^n_\varphi-C_\sigma\leq \int_M\log \frac{e^{\theta_M+V(\varphi)}\omega^n_\varphi}{\mu}\cdot e^{\theta_M+V(\varphi)}\omega^n_\varphi.
$$
By a direct calculation and Proposition \ref{P2.2} we have
\begin{align*}
M_{V,\lambda}(\varphi)&\geq -r(\lambda)\cdot \Big(E_V(\varphi)-\int_M\varphi e^{\theta_M+V(\varphi)} \omega^n_{\varphi}\Big)-\sigma\cdot\int_M\varphi e^{\theta_M+V(\varphi)}\omega^n_\varphi-C_\sigma\\
&\geq -r(\lambda)\cdot (I_V(\varphi)-J_V(\varphi))+\sigma\cdot I_V(\varphi)-C_\sigma\\
&\geq -r(\lambda)\cdot \frac{n+\beta_M-\alpha_M}{n+1+\beta_M-\alpha_M}I_V(\varphi)+\sigma\cdot I_V(\varphi)-C_\sigma\\
&= \Big(\sigma-r(\lambda)\cdot\frac{n+\beta_M-\alpha_M}{n+1+\beta_M-\alpha_M}\Big)I_V(\varphi)-C_\sigma.
\end{align*}
For the second argument, we choose $\sigma>r(\lambda)\cdot \frac{n+\beta_M-\alpha_M}{n+1+\beta_M-\alpha_M}$. According to Proposition \ref{P2.2}, $I_V(\varphi)\geq J_V(\varphi)\cdot \frac{n+1+\beta_M-\alpha_M}{n+\beta_M-\alpha_M}$. Thus
\begin{align*}
M_{V,\lambda}(\varphi)&\geq \Big(\sigma-r(\lambda)\cdot\frac{n+\beta_M-\alpha_M}{n+1+\beta_M-\alpha_M}\Big)
\cdot\frac{n+1+\beta_M-\alpha_M}{n+\beta_M-\alpha_M}J_V(\varphi)-C_\sigma\\
&=\Big(\sigma\cdot\frac{n+1+\beta_M-\alpha_M}{n+\beta_M-\alpha_M}-r(\lambda)\Big)J_V(\varphi)-C_\sigma.
\end{align*}
So the second argument holds.
\end{proof}
For each $\varphi\in \mathcal{E}^1_V(M,\omega_0)$, Theorem 2.18 \cite{BN14} says that there exists a unique $\psi\in \mathcal{E}^1_V(M,\omega_0)$ modulo constants such that
$$
e^{\theta_M+V(\psi)}\omega_\psi^n=\frac{e^{-r(\lambda)\varphi}\mu}{\int_Me^{-r(\lambda)\varphi}d\mu}.
$$
This argument has a connection with the so-called Ricci iteration, which is introduced in \cite{Ru}.
\begin{lemma}\label{L3.3}
For $\varphi,\psi\in \mathcal{E}^1_V(M,\omega_0)$ as above, we have
$$
F_{V,\lambda}(\psi)\leq F_{V,\lambda}(\varphi) \ \  \textrm{and} \ \ M_{V,\lambda}(\psi)\leq F_{V,\lambda}(\varphi).
$$
\end{lemma}
\begin{proof}
By the Proposition 2.15 \cite{BN14}, we see that $F_{V,\lambda}(\varphi+C)=F_{V,\lambda}(\varphi)$ and $M_{V,\lambda}(\varphi+C)=M_{V,\lambda}(\varphi)$. To getting the first inequality we only show that $E_V(\psi)\geq E_V(\varphi)$ by assuming that $\int_M e^{-r(\lambda)\varphi}d\mu=\int_M e^{-r(\lambda)\psi}d\mu=1$. Let $\phi_s:=s\varphi+(1-s)\psi, s\in [0,1]$. Proposition 2.17 \cite{BN14} says that $E_V(\phi_s)$ is concave about $s$. This implies
\begin{align*}
E_V(\varphi)-E_V(\psi)&\leq \int_M(\varphi-\psi)e^{\theta_M+V(\psi)}\omega_\psi^n=\int_M(\varphi-\psi)
e^{-r(\lambda)\varphi}d\mu\\
&=\frac{1}{r(\lambda)}\cdot \int_M\log\frac{\mu_\psi}{\mu_\varphi}d\mu_\varphi\leq \frac{1}{r(\lambda)}\cdot \log\int_Md\mu_\psi=0
\end{align*}
Therefore the first inequality holds. Observing that
\begin{align*}
M_{V,\lambda}(\psi)&=-r(\lambda)\cdot E_V(\psi)+r(\lambda)\cdot \int_M \psi e^{-r(\lambda)\varphi}d\mu+\int_M\log\frac{e^{-r(\lambda)\varphi}\mu}{\mu}e^{-r(\lambda)\varphi}d\mu\\
&=-r(\lambda)\cdot E_V(\psi)-r(\lambda)\cdot\int_M(\varphi-\psi)e^{-r(\lambda)\varphi}d\mu\\
&=-r(\lambda)\cdot E_V(\psi)-r(\lambda)\cdot\int_M(\varphi-\psi)e^{\theta_M+V(\psi)}\omega_\psi^n\\
&\leq -r(\lambda)\cdot E_V(\varphi)=F_{V,\lambda}(\varphi).
\end{align*}
The proof is completed.
\end{proof}
\begin{lemma}\label{L3.4}
Fix $C_1>0$, there exists a constant $C'$ such that the sublevel set $\{\varphi\in \mathcal{E}^1_V(M,\omega_0)|J_V(\varphi)\leq C_1 \ \ \textrm{and}\ \ \sup_M\varphi=0\}$ is contained in $\mathcal{E}^1_{V,C'}(M,\omega_0)$.
\end{lemma}
\begin{proof}
For $\varphi\in \mathcal{E}^1_V(M,\omega_0)$ and $\sup_M\varphi=0$, Lemma 3.45 \cite{Da} says that there exists a constant $A$ such that
$$
\int_M\varphi\omega^n_0\leq \sup_M\varphi\leq \int_M \varphi\omega_0^n+A.
$$
Furthermore we have
$$
\int_M\varphi e^{\theta_M}\omega_0^n\geq e^C\cdot \int_M\varphi\omega_0^n\geq -Ae^C.
$$
Let $C'=C_1+Ae^C$, by the definition of $J_V(\varphi)$ we conclude that $E_V(\varphi)\geq -C'$.
\end{proof}
Given an upper semi-continuous $T$-invariant function $h$, we define
$$
P(h)(x):=\sup\{\psi(x)\in \mathbb{R}|\psi\in \PSH(M,\omega_0)^T\ \ \textrm{and}\ \ \psi\leq h\}.
$$
\begin{remark}\label{r3.3}
If we define $P(h)'(x):=\sup\{\psi(x)\in \mathbb{R}|\psi\in \PSH(M,\omega_0)\ \ \textrm{and}\ \ \psi\leq h\}$, then $P(h)=P(h)'$. In fact, on one hand $P(h)\leq P(h)'$ by the definitions. On the other hand, we denote $\overline{P(h)'}$ by the average of $P(h)'$ along the compact group $T$, then $P(h)'\leq h$ and $\max(P(h)',\overline{P(h)'})\in \PSH(M,\omega_0)$. By the definition of $P(h)'$, we have $\max(P(h)',\overline{P(h)'})=P(h)'=\overline{P(h)'}$.
\end{remark}
Proposition 2.16 \cite{BN14} gives
\begin{lemma}\label{L3.5}
Let $w$ be a non-negative $T$-invariant continuous function and $\varphi\in \mathcal{E}^1_V(M,\omega_0)$. Then we have
$$
\frac{d}{dt}E_V(p(\varphi+tw))\Big|_{t=0}=\int_M we^{\theta_M+V(\varphi)}\omega_\varphi^n.
$$
\end{lemma}
Next we give the main theorems of this section.
\begin{theorem}\label{T3.1}
If the twisted Mabuchi functional $M_{V,\lambda}$ is proper, then there exists $\varphi\in \mathcal{E}^1_V(M,\omega_0)$ solving
$$
e^{\theta_M+V(\varphi)}\omega_\varphi^n=\frac{e^{-r(\lambda)\varphi}\mu}{\int_Me^{-r(\lambda)\varphi}d\mu}.
$$
\end{theorem}
\begin{proof}
By the assumption that $M_{V,\lambda}$ is proper and Lemma \ref{L3.4}, we have
$$
\inf_{\mathcal{E}^1_V(M,\omega_0)}M_{V,\lambda}=\inf_{\mathcal{E}^1_{V,C}(M,\omega_0)}M_{V,\lambda}
$$
where $C$ is a constant as Lemma \ref{L3.4}. It follows from Lemma \ref{L3.1} and Lemma \ref{L3.3} that
$$
\inf_{\mathcal{E}^1_{V,C}(M,\omega_0)}M_{V,\lambda}=\inf_{\mathcal{E}^1_{V,C}(M,\omega_0)}F_{V,\lambda}
=\inf_{\mathcal{E}^1_V(M,\omega_0)}F_{V,\lambda}.
$$
Since $F_{V,\lambda}$ is lower semi-continuous on the compact set $\mathcal{E}^1_{V,C}(M,\omega_0)$, we can find $\varphi\in \mathcal{E}^1_{V,C}(M,\omega_0)$ which minimizes the functional $F_{V,\lambda}$ on $\mathcal{E}^1_V(M,\omega_0)$. Fix an arbitrary non-negative $T$-invariant continuous function $w$ and consider
$$
g(t):=-r(\lambda)\cdot E_V(P(\varphi+tw))-\log\int_Me^{-r(\lambda)(\varphi+tw)}d\mu.
$$
Lemma \ref{L3.5} implies that
$$
\frac{d}{dt}g(t)\big|_{t=0}=-r(\lambda)\cdot \int_M we^{\theta_M+V(\varphi)}\omega^n_\varphi+r(\lambda)\cdot\frac{\int_Mwe^{-r(\lambda)\varphi}d\mu}{\int_Me^{-r(\lambda)\varphi}d\mu}
$$
Now $P(\varphi+tw)\leq \varphi+tw$ gives
$$
g(0)\leq F_{V,\lambda}(P(\varphi+tw))\leq g(t),
$$
since $\varphi$ is a minimizer. Therefore,
$$
\int_M we^{\theta_M+V(\varphi)}\omega^n_\varphi=\frac{\int_Mwe^{-r(\lambda)\varphi}d\mu}{\int_Me^{-r(\lambda)\varphi}d\mu}.
$$
Finally we note that $e^{\theta_M+V(\varphi)}\omega^n_\varphi$ and $e^{-r(\lambda)\varphi}d\mu$ are $T$-invariant measures, so given a continuous function $f$, the integral of $f$ with respect to these measures equal to that of the average of $f$ along the compact group $T$.
\end{proof}
\begin{theorem}\label{T3.2}
Assume that the twisted Mabuchi functional $M_{V,\lambda}$ is proper, then we have
\begin{enumerate}
\item $Aut^0(M,V,\omega_0)={1}$, where $Aut^0(M,V,\omega_0)$ denotes the identity component of automorphism group which preserves the form $\omega_0$ and the holomorphic vector field $V$.
\item $M$ admits a unique twisted K\"{a}hler-Ricci soliton.
\end{enumerate}
\end{theorem}
\begin{proof}
(2) is the direct corollary of (1), Proposition 23 \cite{DS} and Theorem 3.1. Let us prove (1), we follow the argument of \cite{BBEGZ}. There exists a twisted K\"{a}hler-Ricci soliton $\omega$ by Theorem \ref{T3.1}. Let $\gamma$ be a $1$-parameter subgroup of $Aut^0(M,V,\omega_0)$ and observing that $\gamma(s)^*\omega$ is also a twisted K\"{a}hler-Ricci soliton for each $s\in \mathbb{C}$. We assume that $\phi$ is a metric on $\pi^*K_M^{-1}$ with curvature $\omega$ and set $\varphi^s:=\gamma(s)^*\phi-\phi_0$ where $\phi_0$ is a metric with curvature $\omega_0$, then $\varphi(x,s):=\varphi^s(x)$ is a $\Phi^*\omega_0$-psh function on $M\times \mathbb{C}$ such that
$$
(\Phi^*\omega_0+\sqrt{-1}\partial\bar{\partial}\varphi)^{n+1}=0,
$$
where $\Phi$ is a projection from $M\times \mathbb{C}$ to $M$. By Proposition 2.17 \cite{BN14}, $E_V(\varphi^s)$ is harmonic on $\mathbb{C}$, while $\int_M(\varphi^s-\varphi^0)e^{\theta_M+V(\varphi^0)}\omega^n_{\varphi^0}$ is subharmonic since $s\mapsto \varphi^s(x)$ is subharmonic for each $x\in M$. It follows that
$$
E_V(\varphi^0)-E_V(\varphi^s)+\int_M(\varphi^s-\varphi^0)e^{\theta_M+V(\varphi^0)}\omega^n_{\varphi^0}
$$
is subharmonic and bounded on $\mathbb{C}$, hence it is vanishing.

Set
$$
\mu_0=\frac{e^{-r(\lambda)\varphi^0}\mu}{\int_Me^{-r(\lambda)\varphi^0}d\mu} \ \ \textrm{and} \ \
F^0_V(\varphi)=-E_V(\varphi)+\int_M\varphi d\mu_0 \ \ \textrm{for}\ \ \varphi\in \mathcal{E}^1_V(M,\omega_0).
$$
\begin{claim}
Given $\varphi\in \mathcal{E}^1_V(M,\omega_0)$, we have
$$
F^0_V(\varphi)=\inf_{\mathcal{E}^1_V(M,\omega_0)}F_V^0\ \ \textrm{if and only if}\ \ \mu_0=e^{\theta_M+V(\varphi)}\omega_\varphi^n.
$$
\end{claim}
\begin{proof}
If $\mu_0=e^{\theta_M+V(\varphi)}\omega_\varphi^n$, then by the concavity of $E_V$, we have
$$
E_V(\varphi)-\int_M\varphi d\mu_0\geq E_V(\psi)-\int_M\psi d\mu_0
$$
for any $\psi\in \mathcal{E}^1_V(M,\omega_0)$. It follows that
$$
F^0_V(\varphi)=\inf_{\mathcal{E}^1_V(M,\omega_0)}F_V^0.
$$

Conversely we assume that $\varphi$ is the minimizer of $F_V^0$ and consider
$$
g(t):=-E_V(P(\varphi+tw))+\int_M(\varphi+tw)d \mu_0
$$
where $w$ is a non-negative $T$-invariant continuous function. The argument of Theorem \ref{T3.1} implies that
$$
\frac{d}{dt}g(t)\big|_{t=0}=-\int_M we^{\theta_M+V(\varphi)}\omega^n_\varphi+\int_Mwd\mu_0
$$
Since $P(\varphi+tw)\leq \varphi+tw$, we see
$$
g(0)\leq -E_V(P(\varphi+tw))+\int_MP(\varphi+tw)d\mu_0\leq g(t)
$$
which gives
$$
\int_Mwe^{\theta_M+V(\varphi)}\omega_\varphi^n=\int_Mwd\mu_0
$$
\end{proof}

By this claim we know that
$$
e^{\theta_M+V(\varphi^s)}\omega_{\varphi^s}^n=\frac{e^{-r(\lambda)\varphi^0}\mu}{\int_Me^{-r(\lambda)\varphi^0}d\mu}.
$$
According to Theorem 2.18 \cite{BN14}, $\varphi^s=\varphi^0+C_s$ where $C_s$ is a constant dependent of $s$. Hence $\gamma(s)^*\omega=\omega$. The automorphism subgroup $Aut^0(M,V,\omega_0)$ is contained in the compact group of isometries of $\omega$ and hence it is trivial.
\end{proof}
\section{$L^\infty$-bound on the potentials}
Let $\pi:\mathcal{M}\rightarrow \Delta$ be the flat family as section 1. In this section we concern two arguments. One is the existence of twisted K\"{a}hler-Ricci solitons on $M_t$ when $|t|$ and $r(\lambda)=1-(1-\lambda)m$ are sufficiently small. The other is $L^\infty$-estimate (relatively to the ambient Fubini-Study metric) for the potentials of twisted K\"{a}hler-Ricci solitons (if exist) when $|t|$ is small enough. To begin with, we recall some concepts in K\"{a}hler geometry which will be used in this section.

Let $M$ be a smooth Fano manifold, $V$ be a holomorphic vector field belonging to a reductive Lie subalgebra and $T$ be the compact group induced by Im$V$. $M$ is embedded into $\mathbb{CP}^N$ by using the $T$-invariant sections of $K_M^{-m}$ and $\omega_{FS}$ is the Fubini-Study metric. We denote $\omega$ by $\frac{1}{m}\omega_{FS}$ and choose a smooth volume form $\Omega$ such that $\Ric(\Omega)=\omega$ (i.e. $\Ric(\omega)=\omega+\sqrt{-1}\partial\bar{\partial}h$ by the relation $\Omega=e^h\omega^m$). $\theta_M$ is a Hamiltonian function on $M$ defined by $L_V\omega=\sqrt{-1}\partial\bar{\partial}\theta_M$ and $\int_M e^{\theta_M}\omega^n=1$.

The twisted K\"{a}hler-Ricci soliton $\omega_\phi=\omega+\sqrt{-1}\partial\bar{\partial}\phi$ on $M$ is defined as the following equation
$$
\Ric(\omega_\phi)-L_V\omega_\phi=(1-\lambda)\omega_{FS}+r(\lambda)\omega_\phi
$$
where $r(\lambda)=1-(1-\lambda)m$, which is equivalent to the complex Monge-Amp\`{e}re equation
$$
e^{\theta_M+V(\phi)}\omega^n_\phi=\frac{e^{-r(\lambda)\phi}\Omega}{\int_Me^{-r(\lambda)\phi}\Omega}.
$$
Denote $\PSH(M,\omega)$ by the space of $\omega$-psh functions on $M$. For the convenience, we give the various functionals on smooth Fano manifold $M$ as section 2 and 3 (see \cite{TZ00} for a collection of them).
\begin{definition}
For $\phi\in C^\infty(M)\cap \PSH(M,\omega)^T$, we define
\begin{align*}
& I_V(\phi)=\int_M\phi(e^{\theta_M}\omega^n-e^{\theta_M+V(\phi)}\omega_\phi^n),\\
& E_V(\phi)=\int_0^1\int_M\phi e^{\theta_M+s\cdot V(\phi)}\omega^n_{s\phi}\wedge ds,\\
& J_V(\phi)=\int_M\phi e^{\theta_M}\omega^n-E_V(\phi),\\
& F_{V,\lambda}(\phi)=-r(\lambda)\cdot E_V(\phi)-\log\int_M e^{-r(\lambda)\phi}\Omega,\\
& M_{V,\lambda}(\phi)=-r(\lambda)\cdot (I_V(\phi)-J_V(\phi))+\int_M\log\frac{e^{\theta_M+V(\phi)}\omega_\phi^n}{\Omega}e^{\theta_M+V(\phi)}\omega_\phi^n.
\end{align*}
\end{definition}
\begin{remark}\label{r4.1}
Lemma \ref{L3.1} implies $M_{V,\lambda}(\phi)\geq F_{V,\lambda}(\phi)$.
\end{remark}

From now on, we return to the setting of Theorem \ref{T1.1}. Namely, we consider a $\mathbb{Q}$-Gorestein smoothing $\mathcal{M}$ of a $\mathbb{Q}$-Fano variety $M_0$. The notations $V_t$, $T_t,\omega_{FS,t}$, $\omega_t$, $\Omega_t$, $h_t$, $\theta_{M_t}$, $I_{V_t}(\phi_t)$, $E_{V_t}(\phi_t)$, $J_{V_t}(\phi_t)$, $F_{V_t,\lambda}(\phi_t)$ and $M_{V_t,\lambda}(\phi_t)$
on $M_t$ for $t\neq 0$ represent the same meaning as above.
\subsection{Existence of the twisted K\"{a}hler-Ricci solitons with small $r(\lambda)$}
Denote $S$ by the singular set of the central fiber $M_0$ of the flat family $\mathcal{M}$. Let $G:(M_0\backslash S)\times\Delta\rightarrow \mathcal{M}$ be a smooth embedding such that $G_t(M_0\backslash S):=G((M_0\backslash S)\times \{t\})\subset M_t$ and $G_0:M_0\backslash S\rightarrow M_0\backslash S$ is the identity map. We have the following lemma.
\begin{lemma}\label{L4.1}
$\theta_{M_t}\circ G_t$ smoothly converge to $\theta_{M_0}$. In particular, $\lim_{t\rightarrow 0}\beta_{M_t}=\beta_{M_0}$ and $\lim_{t\rightarrow 0}\alpha_{M_t}=\alpha_{M_0}$, where $\beta_{M_t}$ and $\alpha_{M_t}$ denote the maximum and minimum of $\theta_{M_t}$.
\end{lemma}
\begin{proof}
The smooth embedding map $G$ satisfies that $G_t^*\omega_{FS,t}$ $C^\infty$-converges to $\omega_{FS,0}$ ($G_t^*\omega_{t}$ $C^\infty$-converges to $\omega_{0}$), $(G_t^{-1})_*V_t$ $C^\infty$-converges to $V_0$ and $G_t^*J_t$ $C^\infty$-converges to $J_0$, where $J_t$ denotes the complex structure of $M_t$. We also have
$$
L_{(G_t^{-1})_*V_t}G_t^*\omega_t=-d(G_t^*J_t)d(\theta_{M_t}\circ G_t)
$$
due to $L_{V_t}\omega_t=\sqrt{-1}\partial\bar{\partial}\theta_{M_t}$. Furthermore, $\theta_{M_t}\circ G_t$ smoothly converges to $\theta_{M_0}$ on $M_0\backslash S$. The second argument is deduced from the first.
\end{proof}
Next we concern the uniform lower bound of the $\alpha$-invariant $\alpha_{\Omega_t}(\omega_t)$ on $M_t$ for $t\in\Delta^*:=\Delta\backslash \{0\}$.
\begin{lemma}\label{L4.2}
There exists a positive constant $l$ only dependent of the upper bound of $\Vol(M_t)$ such that $\alpha_{\Omega_t}(\omega_t)>l$, where $\Vol(M_t)$ denotes the volume of $M_t$.
\end{lemma}
\begin{proof}
Proposition 2.8 \cite{SSY} implies this argument.
\end{proof}
According to Lemma \ref{L4.1}, Lemma \ref{L4.2}, Lemma \ref{L3.2} and Theorem \ref{T3.2}, we have
\begin{proposition}\label{P4.1}
There exists a number $\underline{\lambda}$ such that for any $\lambda\in (1-m^{-1},\underline{\lambda}]$ and $t\in \Delta$, $M_t$ has a unique twisted K\"{a}hler-Ricci soliton.
\end{proposition}
Note that on each $M_t$ $(t\neq 0)$, we obtain the existence and uniqueness of the twisted K\"{a}hler-Ricci soliton in the sense of definition \ref{D3.1} when $r(\lambda)$ is small enough, but we do not know the regularity about this solution. So we need the following proposition.
\begin{proposition}
Assume that $\omega_\phi=\omega+\sqrt{-1}\partial\bar{\partial}\phi$ is the twisted K\"{a}hler-Ricci soliton in the sense of definition \ref{D3.1} on a smooth Fano manifold $M$, then $\phi$ is smooth.
\end{proposition}
\begin{proof}
Proposition 1.4 \cite{BBEGZ} says that $e^{-r(\lambda)\phi}\in L^p(M)$ for all $p\geq 1$. Lemma 5.1 \cite{Zhu00} and Corollary 5.3 \cite{Zhu00} imply that $|V(\phi)|$ is bounded. So by \cite{EGZ}, $\phi$ is continuous on $M$. We can obtain the Laplacian estimate for $\phi$ according to Proposition 6.1 \cite{Zhu00}. By the standard elliptic regularity theory, $\phi$ is smooth.
\end{proof}
\subsection{Uniform lower bounds on Ding functional and Mabuchi functional}
We define $F'_{V_t,\lambda}$ and $M'_{V_t,\lambda}$ to be the infimum of the twisted Ding functional and Mabuchi functional on $M_t$ with base metric $\omega_t$. If the twisted K\"{a}hler-Ricci soliton $\omega_{\phi_{t,\lambda}}$ exists, then $F'_{V_t,\lambda}$ can be achieved at $\phi_{t,\lambda}$ (c.f. P1006 \cite{DS}). The goal of this subsection is to prove the following theorem.
\begin{theorem}\label{T4.1}
Suppose that for a fixed $\lambda\in(1-m^{-1},1)$, there are twisted K\"{a}hler-Ricci solitons $\omega_{\phi_{t,\lambda}}$ for all $t\in \Delta$. Then we have $\limsup_{t\rightarrow 0}F'_{V_t,\lambda}>-\infty$ and $\limsup_{t\rightarrow 0}M'_{V_t,\lambda}>-\infty$.
\end{theorem}
We first prove the statement about the twisted Ding functional. We follow the argument of \cite{SSY} and \cite{LS}. For $r\in (0,1)$, we denote $\mathcal{M}_r=\mathcal{M}|_{\Delta_r}\subset \mathbb{CP}^N\times\Delta_r$, where $\Delta_r$ is a disc with radius $r$ in $\mathbb{C}$. So $\mathcal{M}_r$ can be viewed as a complex analytic variety with smooth boundary, endowed with a natural K\"{a}hler metric $\mathcal{W}=\omega_t+\sqrt{-1}dt\wedge d\bar{t}$. Let $\omega_{\phi_{t,\lambda}}$ be the unique twisted K\"{a}hler-Ricci soliton on each $M_t$ for $\lambda\in (1-m^{-1},1)$. Define a function $\Psi(t,\cdot):=\phi_{t,\lambda}(\cdot)$ on $\mathcal{M}_r$. We consider the Dirichlet problem for the following homogeneous complex Monge-Amp\`{e}re equation
\begin{equation}\label{e4.1}
\left \{ \begin{array}{ll}
(\mathcal{W}+\sqrt{-1}\partial\bar{\partial}\Phi)^{n+1}=0,\\
\mathcal{W}+\sqrt{-1}\partial\bar{\partial}\Phi\geq 0,\\
\Phi|_{\partial\mathcal{M}_r}=\Psi.
\end{array}
\right.
\end{equation}
Proposition 2.7 \cite{B15} claims that $\Phi:=\sup\{\Phi'\in \PSH(\mathcal{M}_r,\mathcal{W})|\Phi'\leq\Psi \ \ \textrm{on}\ \ \partial\mathcal{M}_r\}$ is the unique solution of the equation (\ref{e4.1}). Note that $\Phi$ is $\mathcal{T}$-invariant by the same argument of Remark \ref{r3.3}, where $\mathcal{T}$ is the compact group induced by Im$\mathcal{V}$.

We need the following auxiliary lemma (c.f. Proposition 2.17 \cite{SSY} and \cite{PS}).
\begin{lemma}
The Dirichlet problem (\ref{e4.1}) has a unique solution which is bounded on $\mathcal{M}_r$ (i.e. $||\Phi||_{L^\infty}\leq C$) and locally $C^{1,\alpha}$ away from the singular set of $\mathcal{M}_r$
\end{lemma}
Denote $\Phi$ by the solution of (\ref{e4.1}). For $t\in \Delta_r$, set
$$
f(t)=-r(\lambda)\cdot E_{V_t}(\Phi_t)=-r(\lambda)\cdot\int_0^1\int_{M_t}\Phi_te^{\theta_{M_t}+s\cdot V_t(\Phi_t)}\omega_{s\Phi_t}^n\wedge ds
$$
and
$$
g(t)=-\log\int_{M_t}e^{-r(\lambda)\Phi_t}\Omega_t
$$
where $\Phi_t=\Phi|_{M_t}$. Then the twisted Ding functional is the sum of these two functions.
\begin{proposition}\label{P4.3}
The function $g(t)$ is continuous and subharmonic on $\Delta_r$.
\end{proposition}
\begin{proof}
From the $C^{1,\alpha}$ regularity of $\Phi$, $g(t)$ is continuous on $\Delta_r^*$. Next we prove that $g$ is subharmonic on $\Delta_r^*$. It suffices to prove this for $t$ in a small disk $\Delta'\subset \Delta^*$. $e^{-r(\lambda)\Phi_t}\Omega_t$ can be viewed as a smooth Hermitian metric on $K^{-1}_{\mathcal{M}/\Delta}$ with positive curvature $r(\lambda)(\mathcal{W}+\sqrt{-1}\partial\bar{\partial}\Phi)+(1-r(\lambda))\mathcal{W}$. Consider the direct image bundle $D$ with fibers $D_t=\Gamma(M_t,K_{M_t}^{-1}\otimes K_{M_t})$, the trivial section $e$ has $L^2$-norm given by
$$
||e||_t^2=\int_{M_t}e^{-r(\lambda)\Phi_t}\Omega_t.
$$
Berndtsson's positivity of the direct image bundle (see Lemma 2.1 \cite{Bo1} and Theorem 3.1 \cite{Bo2}) implies that $-\log||e||_t^2$ is a smooth subharmonic function over $\Delta'$. Note that $g$ is continuous at $t=0$ by the same calculation of Lemma 2 \cite{Li1}. Therefore $g$ is subharmonic on $\Delta_r$.
\end{proof}
\begin{proposition}\label{P4.4}
The function $f(t)$ is continuous on $\Delta_r$.
\end{proposition}
\begin{proof}
We only prove the continuity at $t=0$. For any $\delta>0$, we choose $U_0^\delta$ to be the complement of a small neighborhood $S$ (the singular set of $M_0$) such that
$$
\int_0^1\int_{M_0\backslash U_0^\delta}e^{\theta_{M_0}+s\cdot V_0(\Phi_0)}\omega^n_{s\Phi_0}\wedge ds<\delta.
$$
We then extend $U_0^\delta$ to a smooth family of open subsets $U_t^\delta$ in $M_t$. Using the same notations as the beginning of subsection 4.1, we let $G:U_0^\delta\times \Delta_r\rightarrow \mathcal{M}_r$ be the smooth embedding. By the $C^{1,\alpha}$ regularity of $\Phi$, we see that $\Phi_t\circ G_t$ $C^{1,\alpha}$-converges to $\Phi_0$ and $(G_t^{-1})_*V_t(\Phi_t\circ G_t)$ $C^\alpha$-converges to $V_0(\Phi_0)$ on $U_0^\delta$. By \cite{BT}, $G_t^*\omega_{s\Phi_t}^n$ converges to $\omega_{s\Phi_0}^n$ as currents on $U_0^\delta$. So we have that
\begin{align*}
& \lim_{t\rightarrow 0}\int_0^1\int_{U_t^\delta}\Phi_te^{\theta_{M_t}+s\cdot V_t(\Phi_t)}\omega_{s\Phi_t}^n\wedge ds=\int_0^1\int_{U_0^\delta}\Phi_0e^{\theta_{M_0}+s\cdot V_0(\Phi_0)}\omega_{s\Phi_0}^n\wedge ds,\\
& \lim_{t\rightarrow 0}\int_0^1\int_{U_t^\delta}e^{\theta_{M_t}+s\cdot V_t(\Phi_t)}\omega_{s\Phi_t}^n\wedge ds=\int_0^1\int_{U_0^\delta}e^{\theta_{M_0}+s\cdot V_0(\Phi_0)}\omega_{s\Phi_0}^n\wedge ds,\\
&\lim_{t\rightarrow 0}\int_0^1\int_{U_t^\delta}\omega_{s\Phi_t}^n\wedge ds=\int_0^1\int_{U_0^\delta}\omega_{s\Phi_0}^n\wedge ds.\\
\end{align*}
The following calculation
$$
\Big|\int_0^1\int_{M_t\backslash U_t^\delta}\Phi_te^{\theta_{M_t}+s\cdot V_t(\Phi_t)}\omega_{s\Phi_t}^n\wedge ds\Big|\leq e^C\cdot ||\Phi||_{L^\infty}\cdot \Big(\int_0^1\int_{M_t}\omega_{s\Phi_t}^n\wedge ds-\int_0^1\int_{U_t^\delta}\omega_{s\Phi_t}^n\wedge ds\Big)
$$
implies
$$
\lim_{t\rightarrow 0}\Big|\int_0^1\int_{M_t\backslash U_t^\delta}\Phi_te^{\theta_{M_t}+s\cdot V_t(\Phi_t)}\omega_{s\Phi_t}^n\wedge ds\Big|\leq \delta e^C\cdot ||\Phi||_{L^\infty}.
$$
Let $\delta\rightarrow 0$, we conclude this proposition.
\end{proof}
\begin{proposition}\label{P4.5}
The function $f(t)$ is subharmonic on $\Delta_r$.
\end{proposition}
\begin{proof}
Choose a small dick $\Delta'\subset \Delta_r^*$, we want to show that $f$ is subharmonic on $\Delta'$. Let $h$ be an arbitrary non-negative function supported on $\Delta'$, then we have the following calculation
\begin{align*}
&\int_{\Delta'}f\sqrt{-1}\partial\bar{\partial}h =-r(\lambda) \int_0^1\int_{\pi^{-1}(\Delta')} \pi^*h \sqrt{-1}\partial\bar{\partial}(\Phi e^{\theta_{\mathcal{M}}+s\cdot \mathcal{V}(\Phi)})\mathcal{W}_{s\Phi}^n\wedge ds=\\
&\quad -r(\lambda) \int_0^1\int_{\pi^{-1}(\Delta')} \pi^*h\sqrt{-1}\partial\bar{\partial}\Phi\cdot e^{\theta_{\mathcal{M}}+s\cdot \mathcal{V}(\Phi)}\mathcal{W}_{s\Phi}^n\wedge ds\\
&\quad +r(\lambda) \int_0^1\int_{\pi^{-1}(\Delta')} \pi^*h \sqrt{-1}\bar{\partial}\Phi\wedge\partial(\theta_{\mathcal{M}}+s\cdot \mathcal{V}(\Phi))e^{\theta_{\mathcal{M}}+s\cdot \mathcal{V}(\Phi)}\mathcal{W}_{s\Phi}^n\wedge ds\\
&\quad -r(\lambda) \int_0^1\int_{\pi^{-1}(\Delta')} \pi^*h \sqrt{-1}\partial\Phi\wedge\bar{\partial}(\theta_{\mathcal{M}}+s\cdot \mathcal{V}(\Phi))e^{\theta_{\mathcal{M}}+s\cdot \mathcal{V}(\Phi)}\mathcal{W}_{s\Phi}^n\wedge ds\\
&\quad -r(\lambda) \int_0^1\int_{\pi^{-1}(\Delta')}\pi^*h\cdot\Phi \sqrt{-1}\partial(\theta_{\mathcal{M}}+s\cdot \mathcal{V}(\Phi))\wedge\bar{\partial}(\theta_{\mathcal{M}}+s\cdot \mathcal{V}(\Phi))e^{\theta_{\mathcal{M}}+s\cdot \mathcal{V}(\Phi)}\mathcal{W}_{s\Phi}^n\wedge ds\\
&\quad -r(\lambda) \int_0^1\int_{\pi^{-1}(\Delta')}\pi^*h\cdot\Phi \sqrt{-1}\partial\bar{\partial}(\theta_{\mathcal{M}}+s\cdot \mathcal{V}(\Phi))e^{\theta_{\mathcal{M}}+s\cdot \mathcal{V}(\Phi)}\mathcal{W}_{s\Phi}^n\wedge ds.
\end{align*}
Note that
$$
i_{\mathcal{V}}\mathcal{W}_{s\Phi}=\sqrt{-1}\bar{\partial}(\theta_{\mathcal{M}}+s\cdot \mathcal{V}(\Phi))
$$
and
$$
i_{\bar{\mathcal{V}}}\mathcal{W}_{s\Phi}=-\sqrt{-1}\partial(\theta_{\mathcal{M}}+s\cdot \mathcal{V}(\Phi)).
$$
So we have
\begin{align*}
& \int_0^1\int_{\pi^{-1}(\Delta')} \pi^*h \sqrt{-1}\bar{\partial}\Phi\wedge\partial(\theta_{\mathcal{M}}+s\cdot \mathcal{V}(\Phi))e^{\theta_{\mathcal{M}}+s\cdot \mathcal{V}(\Phi)}\mathcal{W}_{s\Phi}^n\wedge ds\\
&=-\frac{1}{n+1}\int_0^1\int_{\pi^{-1}(\Delta')} \pi^*h\cdot \overline{\mathcal{V}(\Phi)}\cdot e^{\theta_{\mathcal{M}}+s\cdot \mathcal{V}(\Phi)}\mathcal{W}_{s\Phi}^{n+1}\wedge ds
\end{align*}
and
\begin{align*}
&\int_0^1\int_{\pi^{-1}(\Delta')} \pi^*h \sqrt{-1}\partial\Phi\wedge\bar{\partial}(\theta_{\mathcal{M}}+s\cdot \mathcal{V}(\Phi))e^{\theta_{\mathcal{M}}+s\cdot \mathcal{V}(\Phi)}\mathcal{W}_{s\Phi}^n\wedge ds\\
&=\frac{1}{n+1}\int_0^1\int_{\pi^{-1}(\Delta')} \pi^*h\cdot \mathcal{V}(\Phi)\cdot e^{\theta_{\mathcal{M}}+s\cdot \mathcal{V}(\Phi)}\mathcal{W}_{s\Phi}^{n+1}\wedge ds.
\end{align*}
The integral by parts implies that
\begin{align*}
& \quad \int_0^1\int_{\pi^{-1}(\Delta')}\pi^*h\cdot\Phi \sqrt{-1}\partial\bar{\partial}(\theta_{\mathcal{M}}+s\cdot \mathcal{V}(\Phi))e^{\theta_{\mathcal{M}}+s\cdot \mathcal{V}(\Phi)}\mathcal{W}_{s\Phi}^n\wedge ds \\
&= -\int_0^1\int_{\pi^{-1}(\Delta')}(\pi^*h\partial\Phi+\Phi\partial(\pi^*h))\wedge\sqrt{-1}\bar{\partial}(\theta_{\mathcal{M}}+s\cdot \mathcal{V}(\Phi))e^{\theta_{\mathcal{M}}+s\cdot \mathcal{V}(\Phi)}\mathcal{W}_{s\Phi}^n\wedge ds\\
& \quad -\int_0^1\int_{\pi^{-1}(\Delta')}\pi^*h\cdot\Phi \sqrt{-1}\partial(\theta_{\mathcal{M}}+s\cdot \mathcal{V}(\Phi))\wedge\bar{\partial}(\theta_{\mathcal{M}}+s\cdot \mathcal{V}(\Phi))e^{\theta_{\mathcal{M}}+s\cdot \mathcal{V}(\Phi)}\mathcal{W}_{s\Phi}^n\wedge ds\\
&=-\frac{1}{n+1}\int_0^1\int_{\pi^{-1}(\Delta')} \pi^*h\cdot \mathcal{V}(\Phi)\cdot e^{\theta_{\mathcal{M}}+s\cdot \mathcal{V}(\Phi)}\mathcal{W}_{s\Phi}^{n+1}\wedge ds \\
& \quad -\int_0^1\int_{\pi^{-1}(\Delta')}\pi^*h\cdot\Phi \sqrt{-1}\partial(\theta_{\mathcal{M}}+s\cdot \mathcal{V}(\Phi))\wedge\bar{\partial}(\theta_{\mathcal{M}}+s\cdot \mathcal{V}(\Phi))e^{\theta_{\mathcal{M}}+s\cdot \mathcal{V}(\Phi)}\mathcal{W}_{s\Phi}^n\wedge ds
\end{align*}
where the second equality holds due to $\mathcal{V}(\pi^*h)=0$. Therefore, we obtain
\begin{align*}
\int_{\Delta'}f\sqrt{-1}\partial\bar{\partial}h &= -r(\lambda)\int_0^1\int_{\pi^{-1}(\Delta')} \pi^*h \sqrt{-1}\partial\bar{\partial}\Phi \cdot e^{\theta_{\mathcal{M}}+s\cdot \mathcal{V}(\Phi)}\mathcal{W}_{s\Phi}^n\wedge ds \\
& \quad -\frac{r(\lambda)}{n+1}\int_0^1\int_{\pi^{-1}(\Delta')} \pi^*h\cdot \mathcal{V}(\Phi)\cdot e^{\theta_{\mathcal{M}}+s\cdot \mathcal{V}(\Phi)}\mathcal{W}_{s\Phi}^{n+1}\wedge ds\\
&=-\frac{r(\lambda)}{n+1}\int_{\pi^{-1}(\Delta')}\pi^*h\int_0^1\frac{d}{ds}(e^{\theta_{\mathcal{M}}+s\cdot \mathcal{V}(\Phi)}\mathcal{W}_{s\Phi}^{n+1})ds \\
&=-\frac{r(\lambda)}{n+1}\int_{\pi^{-1}(\Delta')}\pi^*h\cdot (e^{\theta_{\mathcal{M}}+\mathcal{V}(\Phi)}\mathcal{W}_{\Phi}^{n+1}-e^{\theta_{\mathcal{M}}}\mathcal{W}^{n+1})\\
&=\frac{r(\lambda)}{n+1}\int_{\pi^{-1}(\Delta')}\pi^*h\cdot e^{\theta_{\mathcal{M}}}\mathcal{W}^{n+1}\geq 0
\end{align*}
Thus, $f$ is subharmonic on $\Delta_r$ since it is continuous at $t=0$.
\end{proof}
Now we give the proof of Theorem \ref{T4.1}.
\begin{proof}[Proof of Theorem \ref{T4.1}]
From Proposition \ref{P4.3}, \ref{P4.4} and \ref{P4.5}, we know that $F_{V_t,\lambda}(\Phi_t)$ is a continuous subharmonic function on $\Delta_r$, so by the maximum principle
$$
\sup_{\Delta_r}F_{V_t,\lambda}(\Phi_t)=\sup_{\partial\Delta_r}F'_{V_t,\lambda}\geq F_{V_0,\lambda}(\Phi_0)\geq F'_{V_0,\lambda},
$$
where the last inequality bases on the fact that the twisted K\"{a}hler-Ricci soliton minimizes the twisted Ding functional. Finally, letting $r\rightarrow 0$, we obtain
$$
\limsup_{t\rightarrow 0}F'_{V_t,\lambda}\geq F'_{V_0,\lambda}.
$$
This proves the statement about the twisted Ding functional. Remark \ref{r4.1} claims that $\limsup_{t\rightarrow 0}M'_{V_t,\lambda}\geq F'_{V_0,\lambda}$, so we complete the proof.
\end{proof}
\subsection{$L^\infty$-estimates and locally higher order estimates for the potentials}
The goal of this subsection is to obtain $L^\infty$-estimates and Laplacian estimates for the potentials. First we establish some auxiliary lemmas. Let $M$ be a smooth Fano manifold, the Aubin's functional are given by
$$
I(\phi)=\int_M\phi(\omega^n-\omega_\phi^n) \ \ \textrm{and} \ \ J(\phi)=\int_0^1\int_M\phi(\omega^n-\omega_{s\phi}^n)\wedge ds
$$
where $\phi\in C^\infty(M)\cap \PSH(M,\omega)^T$.
\begin{lemma}\label{L4.4}
There are positive constants $C_1(\alpha_M)$ and $C_2(\beta_M)$ such that
$$
C_1(I(\phi)-J(\phi))\leq I_V(\phi)-J_V(\phi)\leq C_2(I(\phi)-J(\phi)).
$$
\end{lemma}
\begin{proof}
Take a path $\phi_s=s\phi$, then by Lemma 3.3 \cite{TZ00}, we have
$$
\frac{d}{ds}(I_V(\phi_s)-J_V(\phi_s))=s\int_M|\partial\phi|^2_{\omega_{s\phi}}e^{\theta_M+s\cdot V(\phi)}\omega^n_{s\phi}.
$$
We also know that
$$
\frac{d}{ds}(I(\phi_s)-J(\phi_s))=s\int_M|\partial\phi|^2_{\omega_{s\phi}}\omega^n_{s\phi}.
$$
Thus
$$
e^{\alpha_M}\frac{d}{ds}(I(\phi_s)-J(\phi_s))\leq \frac{d}{ds}(I_V(\phi_s)-J_V(\phi_s))\leq e^{\beta_M}\frac{d}{ds}(I(\phi_s)-J(\phi_s)).
$$
So we obtain this lemma.
\end{proof}
\begin{lemma}\label{L4.5}
For $\lambda\in[\lambda_1,\lambda_2]$ and $1-m^{-1}<\lambda_1<\lambda_2<1$, if $\phi_{t,\lambda}$ are twisted K\"{a}hler-Ricci solitons for $t\in \Delta^*$ and $I_{V_t}(\phi_{t,\lambda})$ is uniformly bounded for $t$ and $\lambda$, then there is a uniform constant $C$ such that
$$
||\phi_{t,\lambda}||_{L^\infty}\leq C
$$
for $t\in \Delta^*$ and $\lambda\in[\lambda_1,\lambda_2]$.
\end{lemma}
\begin{proof}
Let $G(\cdot,\cdot)$ be the Green function of $\frac{1}{m}\omega_{FS}$ for Laplacian operator $\triangle_{\frac{1}{m}\omega_{FS}}$ on $\mathbb{CP}^N$, then there exists a constant $C_1>0$ such that $G(\cdot,\cdot)\geq -C_1$. Denote $G_t(\cdot,\cdot)$ by $G(\cdot,\cdot)|_{M_t\times M_t}$. By the inequality $\triangle_{\omega_t}\phi_{t,\lambda}>-n$, we have
$$
r(\lambda)\Big(\sup_{M_t}\phi_{t,\lambda}-\frac{1}{a_t}\int_{M_t}\phi_{t,\lambda}\omega_t^n\Big)\leq \frac{r(\lambda)}{a_t}\cdot n\cdot \int_{M_t}(G_t(\cdot,\cdot)+C_1)\omega_t^n\leq C_2
$$
where $a_t=\int_{M_t}\omega_t^n$.

Next, set $\omega_{t,\lambda}:=\omega_{\phi_{t,\lambda}}$. Since $\Ric(\omega_{t,\lambda})-\sqrt{-1}\partial\bar{\partial}(\theta_{M_t}+V_t(\phi_{t,\lambda}))>r(\lambda)\omega_{t,\lambda}$,
Theorem B \cite{M2} implies that
$$
G_{\omega_{t,\lambda}}(x,y)\geq -C_3
$$
where $G_{\omega_{t,\lambda}}(\cdot,\cdot)$ is the Green function of $\omega_{t,\lambda}$ for operator $\Re D_{t,\lambda}$ ($D_{t,\lambda}:=\triangle_{\omega_{t,\lambda}}+V_t$). By the inequality $D_{t,\lambda}(\phi_{t,\lambda})\leq n+C$, we have
$$
r(\lambda)\Big(\inf_{M_t}\phi_{t,\lambda}-\frac{1}{a_t}\int_{M_t}\phi_{t,\lambda}\omega_{t,\lambda}^n\Big)\geq -C_4.
$$
Therefore,
$$
\osc \phi_{t,\lambda}:=\sup_{M_t}\phi_{t,\lambda}-\inf_{M_t}\phi_{t,\lambda}\leq C+\frac{1}{a_t}\int_{M_t}\phi_{t,\lambda}(\omega_t^n-\omega^n_{t,\lambda}).
$$
This lemma is followed by the boundedness of $I_{V_t}(\phi_{t,\lambda})$ and Lemma \ref{L4.4}.
\end{proof}
\begin{remark}
If $I(\phi_{t,\lambda})$ is uniformly bounded for $t$ and $\lambda$, then the above lemma can also be deduced.
\end{remark}
Let $K_0$ be a compact subset of $M_0\backslash S$, then we construct a compact subset $K:=\bigcup_{t\in \Delta^*}G_t(K_0)$ on $\mathcal{M}|_{\Delta^*}$, where $G_t$ is explained at the beginning of subsection 4.1. Then we have
\begin{lemma}\label{L4.6}
For $\lambda\in[\lambda_1,\lambda_2]$ and $1-m^{-1}<\lambda_1<\lambda_2<1$, suppose that $\phi_{t,\lambda}$ are twisted K\"{a}hler-Ricci solitons for $t\in \Delta^*$ and $\phi_{t,\lambda}$ is uniformly bounded for $t$ and $\lambda$. If $K$ is a compact subset of $\mathcal{M}|_{\Delta^*}$ as above, then for each $k\geq 2$, we have $||\phi_{t,\lambda}||_{C^k(K\cap M_t)}\leq C$, where $C$ is a positive constant only depending on $\lambda_1$, $\lambda_2$, $K$, $k$ and $||\phi_{t,\lambda}||_{L^\infty}$.
\end{lemma}
\begin{proof}
By a direct calculation, we have
$$
\triangle_{\omega_{t,\lambda}}\log tr_{\omega_t}\omega_{t,\lambda}\geq -\frac{tr_{\omega_t}(\Ric(\omega_{t,\lambda}))}{tr_{\omega_t}\omega_{t,\lambda}}-C_1\cdot tr_{\omega_{t,\lambda}}\omega_t,
$$
where $C_1$ is a constant of the lower bound for the holomorphic bisectional curvature of $\omega_t$. Note that, by the definition of $\phi_{t,\lambda}$,
$$
-\Ric(\omega_{t,\lambda})+\sqrt{-1}\partial\bar{\partial}\theta(\phi_{t,\lambda})=-r(\lambda)\omega_{t,\lambda}-(1-\lambda)\omega_{FS}
$$
where $\theta(\phi_{t,\lambda}):=\theta_{M_t}+V_t(\phi_{t,\lambda})$. Applying the inequality $n\leq (tr_{\omega_t}\omega_{t,\lambda})\cdot(tr_{\omega_{t,\lambda}}\omega_t)$, we get
$$
\triangle_{\omega_{t,\lambda}}\log tr_{\omega_t}\omega_{t,\lambda}\geq -C_2\cdot tr_{\omega_{t,\lambda}}\omega_t-C_3-\frac{\triangle_{\omega_t}\theta(\phi_{t,\lambda})}{tr_{\omega_t}\omega_{t,\lambda}}.
$$
Set $H=\log tr_{\omega_t}\omega_{t,\lambda}-(C_2+1)\phi_{t,\lambda}$, so we have
$$
\triangle_{\omega_{t,\lambda}} H\geq tr_{\omega_{t,\lambda}}\omega_t-C_4-\frac{\triangle_{\omega_t}\theta(\phi_{t,\lambda})}{tr_{\omega_t}\omega_{t,\lambda}}.
$$
Assume that the function $H$ achieves its maximum at some point $x_0$, then at this point
\begin{equation}\label{e4.2}
\nabla\big(e^{-(C_2+1)\phi_{t,\lambda}}(n+\triangle_{\omega_t}\phi_{t,\lambda})\big)=0.
\end{equation}
At $x_0$, we choose the normal coordinate so that $g_{t,i\bar{j}}=\delta_{ij}$ and $(\phi_{t,\lambda})_{i\bar{j}}=\delta_{ij}\cdot (\phi_{t,\lambda})_{i\bar{i}}$. Therefore, (\ref{e4.2}) gives
$$
\big(n+(\phi_{t,\lambda})_{i\bar{i}}\big)_l-\big[(C_2+1)(\phi_{t,\lambda})_l\big](n+\triangle_{\omega_t}\phi_{t,\lambda})=0,
$$
which yields
$$
V_t^l\cdot (\phi_{t,\lambda})_{i\bar{i}l}=V_t^l\cdot \big[(C_2+1)(\phi_{t,\lambda})_l\big](n+\triangle_{\omega_t}\phi_{t,\lambda})=(C_2+1)\cdot V_t(\phi_{t,\lambda})(n+\triangle_{\omega_t}\phi_{t,\lambda}).
$$
Lemma 5.1 \cite{Zhu00} and Corollary 5.3 \cite{Zhu00} imply that $|V_t(\phi_{t,\lambda})|\leq C_5$, so
$$
V_t^l\cdot (\phi_{t,\lambda})_{i\bar{i}l}\leq C_6\cdot(n+\triangle_{\omega_t}\phi_{t,\lambda}).
$$
At point $x_0$, we also have
\begin{align*}
\triangle_{\omega_t}\theta(\phi_{t,\lambda})&= \theta(\phi_{t,\lambda})_{i\bar{i}}=\big[V_t^l\cdot(g_{t,l\bar{i}}+(\phi_{t,\lambda})_{l\bar{i}})\big]_i\\
&=V_t^l\cdot g_{t,l\bar{i}i}+V_t^l\cdot (\phi_{t,\lambda})_{l\bar{i}i}+V_{t,i}^l\cdot (g_{t,l\bar{i}}+(\phi_{t,\lambda})_{l\bar{i}})\\
&\leq C_6\cdot(n+\triangle_{\omega_t}\phi_{t,\lambda})+\sup_{M_t}|V_{t,i}^l|\cdot(n+\triangle_{\omega_t}\phi_{t,\lambda})\\
& \leq C_7\cdot (n+\triangle_{\omega_t}\phi_{t,\lambda})
\end{align*}
Thus we obtain
$$
tr_{\omega_{t,\lambda}}\omega_t(x_0)\leq C_8.
$$
The inequality
\begin{equation}\label{e4.3}
tr_{\omega_t}\omega_{t,\lambda}\leq n\cdot \frac{\omega^n_{t,\lambda}}{\omega^n_t}\cdot (tr_{\omega_{t,\lambda}}\omega_t)^{n-1}
\end{equation}
gives
$$
\log tr_{\omega_t}\omega_{t,\lambda}\leq \log n-\theta(\phi_{t,\lambda})-r(\lambda)\phi_{t,\lambda}+\log\frac{\Omega_t}{\omega^n_t}+(n-1)\log tr_{\omega_{t,\lambda}}\omega_t.
$$
Therefore, the boundedness of $\theta(\phi_{t,\lambda})$ and $\phi_{t,\lambda}$ imply
$$
H\leq H(x_0)\leq C_9+\log\frac{\Omega_t}{\omega^n_t}.
$$
Note that there exists a constant $C'_K$ such that $\frac{\Omega_t}{\omega^n_t}\leq e^{C'_K}$ on $K\cap M_t$, so we have $tr_{\omega_t}\omega_{t,\lambda}\leq C_K$. Using (\ref{e4.3}) again, we get $tr_{\omega_{t,\lambda}}\omega_t\leq\overline{C_K}$. This lemma holds due to the standard Evans-Krylov theory \cite{EV} \cite{K} for the complex Monge-Amp\`{e}re equation.
\end{proof}
The next lemma illustrates that the functional $I$ is continuous under the above continuity of K\"{a}hler potentials (c.f. Lemma 2.14 \cite{SSY}).
\begin{lemma}\label{L4.7}
Suppose that $\phi_{t,\lambda}$ are twisted K\"{a}hler-Ricci solitons and $\phi_{t,\lambda}$ is uniformly bounded. If $\phi_{t,\lambda}\circ G_t$ converges to $\phi_{0,\lambda}$ in the $C^2$ sense on any compact subset away from $S$ on $M_0$, then we have
$$
\lim_{t\rightarrow 0}I(\phi_{t,\lambda})=I(\phi_{0,\lambda}).
$$
\end{lemma}
Fix $\hat{\lambda}\in (0,1-m^{-1})$, by the definition of the twisted Mabuchi functional and the elementary inequality $x\log x\geq -e^{-1}$, we see that for any $\phi\in \PSH(M_t,\omega_t)^{T_t}\cap C^\infty(M_t)$,
\begin{align*}
M_{V_t,\hat{\lambda}}(\phi)&
=-r(\hat{\lambda})\big(I_{V_t}(\phi)-J_{V_t}(\phi)\big)+\int_M\log\frac{e^{\theta_{M_t}+V_t(\phi)}\omega^n_{\phi_t}}{\Omega_t}e^{\theta_{M_t}+V_t(\phi)}\omega^n_{\phi_t}\\
&\geq -r(\hat{\lambda})\big(I_{V_t}(\phi)-J_{V_t}(\phi)\big)-e^{-1}\cdot \int_{M_t}\Omega_t\geq -r(\hat{\lambda})\big(I_{V_t}(\phi)-J_{V_t}(\phi)\big)-C
\end{align*}
where $C$ is a constant independent of $t$ since the volume of $M_t$ can be bounded. On the other hand, assume that there exists a twisted K\"{a}hler-Ricci soliton on each $M_t$ for $t\in \Delta^*$ when $\lambda=\bar{\lambda}$, then by Theorem \ref{T4.1}, we can find $C>0$ and a sequence $t_i\rightarrow 0$ such that $M_{V_i,\bar{\lambda}}(\phi)\geq -C$ $(V_{t_i}=V_i)$ for any $\phi\in \PSH(M_{t_i},\omega_{t_i})^{T_{t_i}}\cap C^\infty(M_{t_i})$.

Note that the twisted Mabuchi functional is linear in $\lambda$, i.e.
$$
s M_{V_i,\hat{\lambda}}(\phi)+(1-s) M_{V_i,\bar{\lambda}}(\phi)=M_{V_i,s\hat{\lambda}+(1-s)\bar{\lambda}}(\phi),
$$
so for each $\lambda\in[\hat{\lambda},\bar{\lambda})$, we have
$$
M_{V_i,\lambda}\geq \delta_\lambda\cdot \big(I_{V_i}(\phi)-J_{V_i}(\phi)\big)-C
$$
where $\delta_\lambda=-r(\hat{\lambda})\frac{\lambda-\bar{\lambda}}{\hat{\lambda}-\bar{\lambda}}>0$. Lemma \ref{L3.1} and \ref{L3.3} claim that the minimizer of the twisted Mabuchi functional is just that of the twisted Ding functional, so the twisted K\"{a}hler-Ricci soliton $\phi_{t,\lambda}$ is the minimum of $M_{V_t,\lambda}$. Therefore for $\lambda\in[\hat{\lambda},\bar{\lambda})$, we have
$$
     \int_{M_i}(\theta_i-h_i)e^{\theta_i}\omega_{t_i}^n=M_{V_i,\lambda}(0)\geq M_{V_i,\lambda}(\phi_{t_i,\lambda})\geq \delta_\lambda\cdot \big(I_{V_i}(\phi_{t_i,\lambda})-J_{V_i}(\phi_{t_i,\lambda})\big)-C
$$
where $\theta_i=\theta_{M_{t_i}}$ and $h_i$ is defined by $\Ric(\omega_{t_i})=\omega_{t_i}+\sqrt{-1}\partial\bar{\partial}h_i$. According to Proposition 2.21 \cite{SSY} and Proposition \ref{P2.2}, we obtain
$$
I_{V_i}(\phi_{t_i,\lambda})\leq C\delta_\lambda^{-1}.
$$

If fix a small number $\epsilon>0$, then for any $\lambda\in[1-m^{-1}+\epsilon,\bar{\lambda}-\epsilon]$, we have $I_{V_i}(\phi_{t_i,\lambda})\leq C'_\epsilon$ which implies $||\phi_{t_i,\lambda}||_{L^\infty}\leq C_\epsilon$ by Lemma \ref{L4.5}. Lemma \ref{L4.6} claims that by passing to a subsequence $t'_i$ ($\{t'_i\}\subset \{t_i\}$), $\phi_{t'_i,\lambda}\circ G_{t'_i}$ $C^\infty$-converges to a smooth function $\phi_{0,\lambda}$ on $M_0\backslash S$, which satisfies
$$
e^{\theta_{M_0}+V_0(\phi_{0,\lambda})}(\omega_0+\sqrt{-1}\partial\bar{\partial}\phi_{0,\lambda})^n=e^{-r(\lambda)\phi_{0,\lambda}}\Omega_0.
$$
This equation implies that $\phi_{0,\lambda}$ is a twisted K\"{a}hler-Ricci soliton on $M_0$. By Theorem \ref{T3.2}, we know that $\phi_{t_i,\lambda}\circ G_{t_i}$ $C^\infty$-converges to $\phi_{0,\lambda}$ on $M_0\backslash S$.
\begin{proposition}\label{P4.6}
For $\lambda\in [1-m^{-1}+\epsilon,\bar{\lambda}-\epsilon]$, suppose that $\phi_{t,\lambda}$ are the twisted K\"{a}hler-Ricci solitons, then
$$
\limsup_{\delta\rightarrow 0}\max_{|t|=\delta}I(\phi_{t,\lambda})<+\infty.
$$
In particular
$$
||\phi_{t,\lambda}||_{L^\infty}<C_\epsilon.
$$
\end{proposition}
\begin{proof}
By the previous argument and Lemma \ref{L4.4}, $I(\phi_{t_i,\lambda})\leq C$. Next we argue by contradiction, then we pick $s_j$ and $|s_j|\rightarrow 0$ such that $I(\phi_{s_j,\lambda})=C+1$. The same argument as above claims that $\phi_{s_j,\lambda}\circ G_{s_j}$ $C^\infty$-converges to $\phi_{0,\lambda}$. Lemma \ref{L4.7} gives
$$
\lim_{j\rightarrow\infty}I(\phi_{s_j,\lambda})=C+1=I(\phi_{0,\lambda})=\lim_{i\rightarrow\infty}I(\phi_{t_i,\lambda})\leq C.
$$
This is a contradiction, so we obtain this consequence.
\end{proof}
Finally, we give the following theorem.
\begin{theorem}\label{T4.2}
For $\lambda\in [1-m^{-1}+\epsilon,\bar{\lambda}-\epsilon]$, suppose that $\phi_{t,\lambda}$ are the twisted K\"{a}hler-Ricci solitons, then $\phi_{t,\lambda}\circ G_t$ $C^\infty$-converges to $\phi_{0,\lambda}$ on $M_0\backslash S$.
\end{theorem}
\section{Gromov-Hausdorff convergence under $L^\infty$-bound on K\"{a}hler potentials }
In this section we investigate the behavior of twisted K\"{a}hler-Ricci solitons in a $\mathbb{Q}$-Gorestein smoothing family. We use the techniques in \cite{WZ} \cite{RZ} \cite{DS} \cite{S2} \cite{DS1} \cite{LY} \cite{LTZ} \cite{WW} \cite{PSS} and \cite{TZ12}, to show that the Gromov-Hausdorff limit as $t\rightarrow 0$ is unique and equal to the twisted K\"{a}hler-Ricci soliton on the central fiber.

First we recall some elementary facts. Suppose that $\omega_{t,\lambda}$ are twisted K\"{a}hler-Ricci solitons on $M_t$ for $t\in \Delta^*$ and $1-m^{-1}+\epsilon\leq\lambda<1$, then Theorem B \cite{M1} says that the diameter of $(M_t,\omega_{t,\lambda})$ has a uniform upper bound only depending on $r(\lambda)$. Lemma 6.1 \cite{WZ} claims that
\begin{equation}\label{e5.1}
|\nabla \theta(\phi_{t,\lambda})|_{\omega_{t,\lambda}}+|\triangle_{\omega_{t,\lambda}}\theta(\phi_{t,\lambda})|\leq C
\end{equation}
where $\theta(\phi_{t,\lambda})=\theta_{M_t}+V_t(\phi_{t,\lambda})$ and $C$ is a positive constant only depending on $|\theta_{\phi_{t,\lambda}}|$ and $|V_t|_{\omega_t}^2$, moreover we can choose $C$ independent of $t$. By Theorem 6.2 \cite{WZ}, we have the non-collapsing property, i.e. for any $p_t\in M_t$,
$$
\Vol\big(B_{p_t}(1),\omega_{t,\lambda}\big)\geq C>0,
$$
where $C$ is a positive constant independent of $t$. By the Gromov precompactness theorem, passing to a subsequence $t_i\rightarrow 0$, we may assume that
$$
(M_{t_i},\omega_{t_i,\lambda})\xrightarrow{d_{GH}} (X,d).
$$
The limit $(X,d)$ is a compact length metric space. It has regular/singular decomposition $X=\mathcal{R}\cup \mathcal{S}$, a point $x\in \mathcal{R}$ if and only if the tangent cone at $x$ is the Euclidean space $\mathbb{R}^{2n}$. To simplify notation, we denote $(M_i,\omega_{i,\lambda})$ by
$(M_{t_i},\omega_{t_i,\lambda})$.
\begin{lemma}
The regular set $\mathcal{R}$ is open in the limit space $(X,d)$.
\end{lemma}
\begin{proof}
If $x\in \mathcal{R}$, then there exists $r=r(x)>0$ such that $\mathcal{H}^{2n}(B_{d}(x,r))\geq (1-\frac{\delta}{2})\Vol(B^0_{r})$, where $\mathcal{H}^{2n}$ denotes the Hausdorff measure and $B^0_r$ is a ball of radius $r$ in $2n$-Euclidean space. Suppose $x_i\in M_i$ satisfying $x_i\xrightarrow{d_{GH}} x$, then by the volume convergence theorem (Remark 5.2 \cite{WZ}), $\Vol(B(x_i,r),\omega_{i,\lambda})\geq (1-\delta)\Vol(B^0_r)$ for sufficiently large i. Proposition 21 \cite{DS} claims that there exists a positive constant $A$ such that $0<\Ric(\omega_{i,\lambda})-L_{V_i}\omega_{i,\lambda}\leq A\omega_{i,\lambda}$ in $B_{\omega_{i,\lambda}}(x_i,r)$. So by the Proposition 19 \cite{DS}, there exists a constant $\delta'$ such that $B_d(x,\delta'r)$ has $C^{1,\alpha}$ harmonic coordinate. This implies $B_d(x,\delta'r)\subset \mathcal{R}$, furthermore $\mathcal{R}$ is open with a $C^{1,\alpha}$ K\"{a}hler metric $\omega_1$ and $\omega_{i,\lambda}$ converges to $\omega_1$ in $C^{1,\alpha}$-topology.
\end{proof}
Since $\mathcal{R}$ is dense in $X$, so we have the following lemma.
\begin{lemma}
$(X,d)=\overline{(\mathcal{R},\omega_1)}$, the metric completion of $(\mathcal{R},\omega_1)$.
\end{lemma}
Next, by the argument of \cite{RZ}, we define $\Gamma_t:=M_t\backslash G_t(S)$, where $S$ denotes the singular set of $M_0$. Define the Gromov-Hausdorff limit of $\Gamma_t$
$$
\Gamma:=\big\{x\in X| \textrm{there exists $x_i\in\Gamma_i:=\Gamma_{t_i}$ such that $x_i\rightarrow x$} \big\}.
$$
Assume that the K\"{a}hler potentials $||\phi_{t,\lambda}||_{L^\infty}$ is uniformly bounded, then we have
\begin{proposition}\label{P5.1}
$(X,d)$ is isometric to $\overline{(M_0\backslash S,\omega_{0,\lambda})}$, where $\omega_{0,\lambda}$ is the unique twisted K\"{a}hler-Ricci soliton on $M_0$.
\end{proposition}
\begin{proof}
First we prove the following claim.
\begin{claim}
$\Gamma\backslash \mathcal{S}$ is a subvariety of dimension $(n-1)$ if it is not empty.
\end{claim}
\begin{proof}
Let $x\in \Gamma\backslash \mathcal{S}$ and $x_i\in \Gamma_i$ such that $ x_i\xrightarrow{d_{GH}} x$. By the $C^{1,\alpha}$ convergence of $\omega_{i,\lambda}$, there are $C,r>0$ independent of $i$ and a sequence of harmonic coordinates in $B_{\omega_{i,\lambda}}(x_i,r)$ such that $C^{-1}\omega_E\leq \omega_{i,\lambda}\leq C\omega_E$ where $\omega_E$ is the Euclidean metric in the coordinates. Furthermore, the sequence of harmonic coordinates can be perturbed to holomorphic coordinates on $B_{\omega_{i,\lambda}}(x_i,r)$ according to Lemma 3.11 \cite{TZ16}. Since the total volume of
$\Gamma_i$ is uniformly bounded, the local analytic set $\Gamma_i\cap B_{\omega_{i,\lambda}}(x_i,r)$ have a uniform bound of degree and so converge to an analytic set $\Gamma\cap B_d(x,r)$.
\end{proof}
From the above claim we know that $\dim_{\mathbb{R}}(\Gamma)\leq 2n-2$. By the argument of \cite{RZ}, $(X\backslash \Gamma,\omega_1)$ is homeomorphic and locally isometric to $(M_0\backslash S,\omega_{0,\lambda})$. Since $X$ is a length metric space and $\dim_{\mathbb{R}}(\Gamma)\leq 2n-2$, $(X\backslash \Gamma,\omega_1)$ is isometric to $(M_0\backslash S,\omega_{0,\lambda})$. So we have
$$
(X,d)=\overline{(X\backslash \Gamma,\omega_1)}=\overline{(M_0\backslash S,\omega_{0,\lambda})}.
$$
\end{proof}
A direct corollary is
\begin{corollary}
$(M_t,\omega_{t,\lambda})$ converges globally to $(X,d)$ in the Gromov-Hausdorff topology as $t\rightarrow 0$.
\end{corollary}
\begin{proposition}
$M_0\backslash S=\mathcal{R}$, the regular set of $X$.
\end{proposition}
\begin{proof}
Since $M_0\backslash S$ has smooth structure in $X$, we have $M_0\backslash S\subset\mathcal{R}$. Next we show the converse. We argue by contradiction. Suppose $p\in \mathcal{R}\backslash (M_0\backslash S)$, then there exists a sequence of points $p_t\in \Gamma_t$ such that $p_t\xrightarrow{d_{GH}} p$. By the $C^{1,\alpha}$ regularity of $(\mathcal{R},\omega_1)$, there exist $C,r>0$ independent of $t$ and a sequence of holomorphic coordinates on $B_{\omega_{t,\lambda}}(p_t,r)$ such that $C^{-1}\omega_E\leq \omega_{t,\lambda}\leq C\omega_E$. Denote $q=\dim_{\mathbb{C}}(\Gamma_t)$, then
$$
\Vol\big(\Gamma_t\cap B_{\omega_{t,\lambda}}(p_t,r)\big)=\int_{\Gamma_t\cap B_{\omega_{t,\lambda}}(p_t,r)}\omega_{t,\lambda}^q\geq
\int_{\Gamma_t\cap B_{\omega_E}(C^{-\frac{1}{2}}r)}(C^{-1}\omega_E)^q
$$
which has a positive lower bound. However this contradicts with the following argument
$$
\Vol\big(\Gamma_t\cap B_{\omega_{t,\lambda}}(p_t,r)\big)\leq \int_{\Gamma_t}\omega_{t,\lambda}^q=\int_{\Gamma_t}\omega_t^q
$$
which tends to $0$ as $t\rightarrow 0$.
\end{proof}
Next we will obtain some uniform $L^2$-estimates for $H^0(M_t,K_{M_t}^{-m})$. For a fixed $\lambda$, using the same notations in \cite{DS1}, we denote
$$
K_{M_t}^\sharp=K_{M_t}^{-m}, \ \  h_{t,\lambda}^\sharp=h_{t,\lambda}^m,  \ \ \omega_{t,\lambda}^\sharp=m\cdot \omega_{t,\lambda},   \ \ L^{p,\sharp}(M_t)=L^p(M_t,\omega_{t,\lambda}^\sharp),
$$
where $\omega_{t,\lambda}$ is twisted K\"{a}hler-Ricci soliton on each $M_t$ and $h_{t,\lambda}$ is the Hermitian metric on $K_{M_t}^{-1}$ with its curvature $\Ric(h_{t,\lambda})=r(\lambda)\omega_{t,\lambda}+(1-\lambda)\omega_{FS}$, i.e. $h_{t,\lambda}=e^{\theta_{t,\lambda}}\omega_{t,\lambda}^n$, where $\theta(\phi_{t,\lambda})=\theta_{M_t}+V_t(\phi_{t,\lambda})$. Let $\widetilde{g_{t,\lambda}}=e^{-\frac{1}{n-1}\theta(\phi_{t,\lambda})}g_{t,\lambda}$, then the estimate (\ref{e5.1}) implies that the Ricci curvature of $\widetilde{g_{t,\lambda}}$ has a uniform lower bound. Therefore, the Sobolev constant is uniform bounded for $\widetilde{g_{t,\lambda}}$, so it is for $g_{t,\lambda}$ as $\widetilde{g_{t,\lambda}}$ and $g_{t,\lambda}$ are uniformly equivalent. The same argument of Proposition 4.1 \cite{PSS} gives
\begin{enumerate}
\item Let $s$ be a holomorphic section of $H^0(M_t,K_{M_t}^\sharp)$, then there exist two constants $C1$, $C_2$ independent of $t$ such that
    $$
    ||s||_{L^{\infty,\sharp}}\leq C_1||s||_{L^{2,\sharp}} \ \ \textrm{and} \ \ ||\nabla s||_{L^{\infty,\sharp}}\leq C_2||s||_{L^{2,\sharp}}.
    $$
\item Assume that $\sigma$ is a $K_{M_t}^\sharp$-valued $(0,1)$-form and its $L^2$ inner product is defined by
    $$
    \int_{M_t}|\sigma|^2_{h_{t,\lambda}^\sharp,g_{t,\lambda}^\sharp}e^{\theta(\phi_{t,\lambda})}(\omega_{t,\lambda}^\sharp)^n.
    $$
    Denote $\bar{\partial}^*_{\theta(\phi_{t,\lambda})}$ by the adjoint operator of $\bar{\partial}$, then there exists a constant $A$ independent of $t$ such that $\bar{\partial}\bar{\partial}^*_{\theta(\phi_{t,\lambda})}+\bar{\partial}^*_{\theta(\phi_{t,\lambda})}\bar{\partial}\geq A$.
\end{enumerate}

The next definition comes from \cite{CDS2} and \cite{S1}.
\begin{definition}
Let $p\in X$ and $C(Y)$ be the tangent cone at $p$. We say that the tangent cone is good if the following hold:
\begin{enumerate}
\item the regular set $Y_{reg}$ is open in $Y$ and smooth,
\item the distance function on $C(Y_{reg})$ is induced by a Ricci flat K\"{a}hler metric,
\item for all $\eta>0$, there is a Lipschitz function $g$ on $Y$, equal to $1$ on a neighborhood of the singular set $S_Y\subset Y$, supported on the $\eta$-neighborhood of $S_Y$ and with $||\nabla g||_{L^2}\leq \eta$.
\end{enumerate}
\end{definition}
The argument of \cite{DS} (P1001) claims that all the tangent cones are good. So by the argument of \cite{DS1}, we have
\begin{theorem}\label{T5.1}
Let $\pi:\mathcal{M}\rightarrow\Delta$ be a $\mathbb{Q}$-Gorestein smoothing family and $\mathcal{V}$ be a reductive vector field on $\mathcal{M}$, which preserves the fibers. For $\lambda\in (1-m^{-1},1]$ there is a twisted K\"{a}hler-Ricci soliton $\omega_{t,\lambda}:=\omega_t+\sqrt{-1}\partial\bar{\partial}\phi_{t,\lambda}$ on each $M_t$ for $t\in \Delta$ with uniformly bounded $||\phi_{t,\lambda}||_{L^\infty}$ $(t\neq 0)$. Then $(M_t,\omega_{t,\lambda})$ converges to $(M_0,\omega_{0,\lambda})$ in the Gromov-Hausdorff topology as $t\rightarrow 0$.
\end{theorem}
\begin{remark}
The same conclusion is true if $\lambda$ vary and stay bounded, i.e. $\lambda\in [\lambda_1,\lambda_2]$ where $1-m^{-1}<\lambda_1<\lambda_2<1$.
\end{remark}
\section{Existence of K\"{a}hler-Ricci solitons}
In this section we show the main theorem of this article by using the argument of section 4 \cite{SSY}. We define the following function:
$$
\lambda_t:=\sup\Big\{\lambda\in (1-\frac{1}{m},1]\Big|\exists \textrm{twisted K\"{a}hler-Ricci solitons on $M_t$ for all $\kappa\leq \lambda$} \Big\}.
$$
\begin{proposition}\label{P6.1}
If $(M_0,V_0)$ is K-stable, then the function $\lambda_t$ is lower semi-continuous on $\Delta$.
\end{proposition}
\begin{proof}
We only deal with the lower semi-continuous at $t=0$ and the other case is easier by the same argument. Suppose that $\lambda_t$ is not lower semi-continuous at $t=0$, i.e. $\liminf_{t\rightarrow 0}\lambda_t=\lambda_\infty<\lambda_0\leq 1$. Choosing an increasing sequence $\lambda_i<\lambda_\infty$ with $\lim_{i\rightarrow \infty}\lambda_i=\lambda_\infty$. For any $i$, the definition of $\lambda_t$ implies that there exists a twisted K\"{a}hler-Ricci soliton $\omega_{t,\lambda_i}$ on each $M_t$ when $|t|$ is small enough. There is a twisted K\"{a}hler-Ricci soliton $\omega_{0,\lambda}$ on $M_0$ for each $\lambda\in [\underline{\lambda},\lambda_0)$, where $\underline{\lambda}$ is defined in Proposition \ref{P4.1}. According to Theorem \ref{T5.1}, for each $i$, $(M_t,V_t,(1-\lambda_i)\omega_{FS},\omega_{t,\lambda_i})$ converges to $(M_0,V_0,(1-\lambda_i)\omega_{FS},\omega_{0,\lambda_i})$ in the Gromov-Hausdorff topology. Then, using the diagonal arguments as section 3 in \cite{DS}, by passing to a subsequence we have that $(M_0,V_0,(1-\lambda_i)\omega_{FS},\omega_{0,\lambda_i})$ converges to $(Y,\widetilde{V},(1-\lambda_\infty)\beta,\omega)$ in the Gromov-Hausdorff topology as $\lambda_i\rightarrow\lambda_\infty$, where $Y$ is a $\mathbb{Q}$-Fano variety, $\widetilde{V}$ is a holomorphic vector field, $\beta$ is a closed positive $(1,1)$-form and $\omega$ is a twisted K\"{a}hler-Ricci soliton. Note that $(M_0,V_0)$ is K-stable, so section 3 \cite{DS} (P991-992) implies that $(M_0,V_0,(1-\lambda_\infty)\omega_{FS})\cong (Y,\widetilde{V},(1-\lambda_\infty)\beta)$. Theorem \ref{T3.2} claims that $(M_0,V_0,(1-\lambda_\infty)\omega_{FS},\omega_{0,\lambda_\infty})\cong (Y,\widetilde{V},(1-\lambda_\infty)\beta,\omega)$.

Let $\mathcal{Z}$ be the space of all $(M_t,V_t,(1-\lambda)\omega_{FS},\omega_{t,\lambda})$ with $\lambda\in[\underline{\lambda},\lambda_t)$. Denote $\overline{\mathcal{Z}}$ by the closure of $\mathcal{Z}$ under the Gromov-Hausdorff convergence and $\mathcal{C}$ by the subspace of $\overline{\mathcal{Z}}\backslash \mathcal{Z}$ which consists of limits of some sequnece $(M_{t_i},V_{t_i},(1-\lambda_i)\omega_{FS},\omega_{t_i,\lambda_i})$ with $t_i\rightarrow 0$ and $\lambda_i\rightarrow\lambda_\infty$. By the argument of \cite{DS1}, we have an injective continuous map from $\overline{\mathcal{Z}}$ into $\mathrm{Ch}/U(N)$, where $\mathrm{Ch}$ denotes the Chow variety. We observe that $(M_0,V_0,(1-\lambda_\infty)\omega_{FS},\omega_{0,\lambda_\infty})$ is in $\mathcal{C}$.
\begin{lemma}\label{L6.1}
We have
$$
\mathcal{C}=\big\{(M_0,V_0,(1-\lambda_\infty)\omega_{FS},\omega_{0,\lambda_\infty})\big\}
$$
\end{lemma}
\begin{proof}
First, we claim that there is an open neighborhood $\mathcal{U}$ of $(M_0,V_0,(1-\lambda_\infty)\omega_{FS},\omega_{0,\lambda_\infty})$ such that $\mathcal{C}\cap\mathcal{U}=\big\{(M_0,V_0,(1-\lambda_\infty)\omega_{FS},\omega_{0,\lambda_\infty})\big\}$. Otherwise, we can choose a sequence $\big\{(Y^i,V_{Y^i},(1-\lambda_\infty)\beta_{Y^i},\omega_{Y^i})\big\}^\infty_{i=1}\subset \mathcal{C}$ converging in the Gromov-Hausdorff topology to $(M_0,V_0,(1-\lambda_\infty)\omega_{FS},\omega_{0,\lambda_\infty})$. Take some sequence $(M_{t_j^i},V_{t_j^i},(1-\lambda_j)\omega_{FS},\omega_{t_j^i,\lambda_j})$ such that $(Y^i,V_{Y^i},(1-\lambda_\infty)\beta_{Y^i},\omega_{Y^i})$ is the Gromov-Hausdorff limit as $t_j^i\rightarrow 0$ and $\lambda_j\rightarrow \lambda_\infty$ for each $i$. For any sequence $\{t^k_{j_k}\}_{k=1}^\infty$ converging to $0$, we have that $(M_{t^k_{j_k}},V_{t^k_{j_k}},(1-\lambda_{j_k})\omega_{FS},\omega_{t_{j_k}^k,\lambda_{j_k}})$ converges to $(M_0,V_0,(1-\lambda_\infty)\omega_{FS},\omega_{0,\lambda_\infty})$, which implies the functional $I(\phi_{t_{j_k}^k,\lambda_{j_k}})\rightarrow I(\phi_{0,\lambda_\infty})$ as $k\rightarrow \infty$. Thus, $\phi_{t_{j}^i,\lambda_{j}}$ is uniformly bounded in $L^\infty$ for all $i$ and $j$. Furthermore, $(M_0,V_0,(1-\lambda_\infty)\omega_{FS},\omega_{0,\lambda_\infty})\cong (Y^i,V_{Y^i},(1-\lambda_\infty)\beta_{Y^i},\omega_{Y^i})$ for each $i$ according to Theorem \ref{T5.1}.

Second, define a family $\mathcal{C}_\alpha:=\bigcup_{0<|t|<\alpha}\big\{(M_t,V_t,(1-\lambda)\omega_{FS},\omega_{t,\lambda})|\lambda\in(\lambda_t-\alpha,\lambda_t)\big\}$
indexed by $\alpha\in(0,1)$, which is precompact in the Gromov-Hausdorff topology. Clearly, $\lim_{\alpha\rightarrow 0}\mathcal{C}_\alpha=\mathcal{C}$ and each $\mathcal{C}_\alpha$ is path-connected. Lemma 4.3 \cite{SSY} claims that $\mathcal{C}$ is connected. So we complete the proof of this lemma.
\end{proof}

By the definition of $\lambda_t$, we let $\lambda$ tends to $\lambda_t$, then by \cite{DS}, $(M_t,V_t,(1-\lambda)\omega_{FS},\omega_{t,\lambda})$ converges by subsequence to some limit $(X_t,\widetilde{V_t},(1-\lambda_t)\beta_t,\overline{\omega_t})$ such that  $\Aut^0(X_t,\widetilde{V_t},(1-\lambda_t)\beta_t,\overline{\omega_t})$ is non-trivial. Choose $\lambda_{t_i}\rightarrow\lambda_\infty$ and $t_i\rightarrow 0$ $(\liminf_{t\rightarrow 0}\lambda_t=\lambda_\infty)$, the limiting sequence $(X_{t_i},\widetilde{V_{t_i}},(1-\lambda_{t_i})\beta_{t_i},\overline{\omega_{t_i}})$ converges by subsequence to $(M_0,V_0,(1-\lambda_\infty)\omega_{FS},\omega_{0,\lambda_\infty})$ due to the structure of $\mathcal{C}$. This is a contradiction with Lemma \ref{L6.1}.
\end{proof}
Next we prove the main theorem of this article.
\begin{theorem}
Suppose that $(M_0,V_0)$ is K-stable, then there exists a K\"{a}hler-Ricci soliton on $M_0$.
\end{theorem}
\begin{proof}
We define a set
$$
\Lambda:=\{\lambda\leq 1|\textrm{there exists a twisted K\"{a}hler-Ricci soliton on $M_0$ for each $\kappa\leq\lambda$}\}.
$$
By Proposition \ref{P4.1}, it suffices to show that $\Lambda$ is both open and closed in $[\underline{\lambda},1]$.

First, we prove the openness. For any $\lambda\in\Lambda$, by the definition of $\Lambda$, we have a twisted K\"{a}hler-Ricci soliton on $M_0$ for each $\kappa\leq\lambda$, so $\lambda<\lambda_0$. Thus, $\lambda_t>\lambda$ for $|t|$ small enough since $\lambda_t$ is lower semi-continuous. We can choose a number $\tilde{\lambda}$ such that $\lambda_t>\tilde{\lambda}>\lambda$ for $|t|$ small enough. For $\lambda'\in [\lambda,\tilde{\lambda}]$, the arguments of section 4 and 5 imply that $(M_t,V_t,(1-\lambda')\omega_{FS},\omega_{t,\lambda'})$ converges to $(M_0,V_0,(1-\lambda')\omega_{FS},\omega_{0,\lambda'})$. Thus, $\Lambda$ is open.

Second, we prove the closedness. Take any sequence $\{\lambda_i\}_{i=1}^\infty\subset \Lambda$ which strictly increases to $\lambda_\infty$. Since $\lambda_t$ is lower semi-continuous, for any $i$, $\lambda_t>\lambda_i$ when $|t|$ is small enough. Furthermore, $(M_0,V_0,(1-\lambda_i)\omega_{FS},\omega_{0,\lambda_i})$ is the Gromov-Hausdorff limit of $(M_t,V_t,(1-\lambda_i)\omega_{FS},\omega_{t,\lambda_i})$ as $t\rightarrow 0$ by Theorem \ref{T5.1}. The diagonal argument claims that by passing to a subsequence we have that $(M_0,V_0,(1-\lambda_i)\omega_{FS},\omega_{0,\lambda_i})$ converges to $(Y,\widetilde{V},(1-\lambda_\infty)\beta,\omega)$ as $\lambda_i\rightarrow \lambda_\infty$. The condition that $(M_0,V_0)$ is K-stable gives that $(Y,\widetilde{V},(1-\lambda_\infty)\beta,\omega)\cong (M_0,V_0,(1-\lambda_\infty)\omega_{FS},\omega_{0,\lambda_\infty})$ according to the argument of \cite{DS} (P991-992). This implies that $\Lambda$ is closed.
\end{proof}
\bigskip

\noindent {\bf{Acknowledgements:}} The author was supported by a grant from the Fundamental Research Funds for the Central Universities. The author also thanks Yi Yao for his enthusiastic discussion.

\end{document}